\documentclass{article}[12pt]

\usepackage{a4wide}
\usepackage[utf8]{inputenc}   
\usepackage[T1]{fontenc}      
\usepackage{multicol}
\usepackage{amsmath,amsthm} 
\usepackage{amssymb,mathrsfs} 
\usepackage{amssymb}
\usepackage{amsfonts}
\usepackage[normalem]{ulem}
\usepackage{nicefrac}
\usepackage{geometry}
\usepackage{graphicx}
\usepackage{enumerate}
\usepackage{graphicx} 
\usepackage{lmodern} 
\usepackage{mathtools}
\usepackage{physics}
\usepackage{url}
\usepackage{color}

\newcommand{\dt}{\dd{t}}
\newcommand{\paq}{\partial_q}
\newcommand{\pap}{\partial_p}

\newcommand{\papi}{\partial_{p_i}}
\newcommand{\papj}{\partial_{p_j}}

\newcommand{\half}{{\nicefrac{1}{2}}}

\newcommand{\ie}{\textit{ie.} }

\newcommand{\calA}{\mathcal{A}}
\newcommand{\calB}{\mathcal{B}}
\newcommand{\calC}{\mathcal{C}}
\newcommand{\calD}{\mathcal{D}}
\newcommand{\calE}{{\calD \times \mathbb{R}^{D}}}

\newcommand{\calH}{\mathcal{H}}

\newcommand{\calL}{\mathcal{L}}

\newcommand{\bbN}{\mathbb{N}}
\newcommand{\bbZ}{\mathbb{Z}}
\newcommand{\bbR}{\mathbb{R}}

\newcommand{\bbT}{\mathbb{T}}

\newcommand{\bbRD}{{\mathbb{R}^D}}

\newcommand{\Esp}{\mathbb{E}}
\newcommand{\Espmu}{\mathbb{E}_\mu}

\newcommand{\rme}{\mathrm{e}}

\newcommand{\rmH}{\mathrm{H}}
\newcommand{\rmL}{\mathrm{L}}

\newcommand{\rmDom}{\mathrm{Dom}}

\newcommand{\bfI}{\mathbf{I}}

\newcommand{\bfL}{\mathbf{L}}

\newcommand{\bfN}{\mathbf{N}}
\newcommand{\bfP}{\mathbf{P}}
\newcommand{\bfQ}{\mathbf{Q}}

\newcommand{\bfS}{\mathbf{S}}

\newcommand{\bfU}{\mathbf{U}}
\newcommand{\bfX}{\mathbf{X}}
\newcommand{\bfY}{\mathbf{Y}}

\newcommand{\bfone}{\mathbf{1}}

\newcommand{\scrC}{\mathscr{C}}
\newcommand{\scrH}{\mathscr{H}}
\newcommand{\naq}{\nabla_q}
\newcommand{\nap}{\nabla_p}

\newcommand{\Lham}{\calL_{\rm{ham}}}

\newcommand{\LFD}{{\calL_\mathrm{FD}}}

\newcommand{\Lovd}{{\calL_\mathrm{ovd}}}
\newcommand{\tLovd}{{\widetilde \calL_\mathrm{ovd}}}

\newcommand{\tcalLM}{\widetilde \calL_M}
\newcommand{\calLinv}{{\calL^{-1}}}

\newcommand{\Zinvnu}{Z_{\beta,\nu}^{-1}}
\newcommand{\Zinvmu}{Z_{\beta,\mu}^{-1}}

\newcommand{\Cnu}{C_\nu}

\newcommand{\lambdaham}{\lambda_{\rm{ham}}}

\newcommand{\core}{\scrC}

\newcommand{\Lmu}{{\rmL^2(\mu)}}
\newcommand{\tLmu}{{\rmL_0^2(\mu)}}

\newcommand{\Hmu}{{\rmH^1(\mu)}}

\newcommand{\Hsmu}{{\rmH^s(\mu)}}
\newcommand{\Htwomu}{{\rmH^2(\mu)}}
\newcommand{\Hnu}{{\rmH^1(\nu)}}

\newcommand{\Hsnu}{{\rmH^s(\nu)}}
\newcommand{\Hskappa}{{\rmH^s(\kappa)}}
\newcommand{\Lnu}{{\rmL^2(\nu)}}

\newcommand{\Lkappa}{{\rmL^2(\kappa)}}

\newcommand{\calBLmu}{{\calBLmu}}
\newcommand{\calBLnu}{{\calBLnu}}
\newcommand{\calBLkappa}{{\calBLkappa}}
\newcommand{\lang}{\left\langle}
\newcommand{\rang}{\right\rangle}

\newcommand{\betainv}{\beta^{-1}}
\newcommand{\betainvinv}{\beta^{-2}}

\newcommand{\Pinot}{\Pi_0}

\newcommand{\Pip}{\Pi_p}

\newcommand{\Picap}{\Pi_{M,0}}
\newcommand{\PiKL}{\Pi_{KL}}
\newcommand{\PiKLnot}{\Pi_{KL,0}}
\newcommand{\PiK}{\Pi_K^q}
\newcommand{\PiKm}{\Pi_{K-1}^{q}}
\newcommand{\PiKo}{\Pi_K^{q\perp}}
\newcommand{\PiKmo}{\Pi_{K-1}^{q\perp}}

\newcommand{\PiKpo}{\Pi_{K+1}^{q\perp}}
\newcommand{\PiL}{\Pi_L^p}
\newcommand{\PiLm}{\Pi_{L-1}^p}

\newcommand{\PiLo}{\Pi_L^{p\perp}}
\newcommand{\PiLmo}{\Pi_{L-1}^{p\perp}}

\newcommand{\PiLp}{\Pi_{L+1}^p}
\newcommand{\Pif}{{D^{+-}_K}}

\newcommand{\VKL}{V_{KL}}

\newcommand{\PhiKL}{\Phi_{KL}}
\newcommand{\PiM}{\Pi_M}

\newcommand{\VM}{V_M}
\newcommand{\PhiM}{\Phi_M}

\newcommand{\Vcap}{V_{M,0}}

\newcommand{\lambdagamma}{\lambda_\gamma}
\newcommand{\lambdagammaM}{\lambda_{\gamma, M}}
\newcommand{\tlambdagammaM}{\widehat{\lambda}_{\gamma, M}}

\newcommand{\tlambdagammaKL}{\widehat{\lambda}_{\gamma, KL}}
\renewcommand{\leq}{\leqslant}

\renewcommand{\geq}{\geqslant}

\newcommand{\scrD}{\mathscr{D}}
\newcommand{\dps}{\displaystyle}

\newtheorem{theorem}{Theorem}
\newtheorem{lemma}{Lemma}
\newtheorem{assumption}{Assumption}
\newtheorem{prop}{Proposition}
\newtheorem{corollary}{Corollary}
\newtheorem{remark}{Remark}
\newtheorem{definition}{Definition}

\begin{document}

\title{Spectral methods for Langevin dynamics and associated error estimates}
\author{Julien Roussel and Gabriel Stoltz\\
\small Université Paris-Est, CERMICS (ENPC), Inria, F-77455 Marne-la-Vallée, France \\
}

\date{}

\maketitle

\abstract{
We prove the consistency of Galerkin methods to solve Poisson equations where the differential operator under consideration is hypocoercive. We show in particular how the hypocoercive nature of the generator associated with Langevin dynamics can be used at the discrete level to first prove the invertibility of the rigidity matrix, and next provide error bounds on the approximation of the solution of the Poisson equation. We present general convergence results in an abstract setting, as well as explicit convergence rates for a simple example discretized using a tensor basis. Our theoretical findings are illustrated by numerical simulations.
}


\section{Introduction}

Statistical physics gives a theoretical framework to bridge the gap between microscopic and macroscopic descriptions of matter~\cite{Balian07}. This is done in practice with numerical methods known as molecular simulation~\cite{Allen87,Frenkel02,Tuckerman10,Leimkuhler16_book}. Despite its intrinsic limitations on spatial and timescales, molecular simulation has been used and developed over the past 50~years, and recently gained some recognition through the 2013 Chemistry Nobel Prize. One important aim of molecular dynamics is to quantitatively evaluate macroscopic properties of interest, obtained as averages of functions of the full microstate of the system (positions and velocities of all atoms in the system) with respect to some probability measure, called thermodynamic ensemble. Some properties of interest are static (a.k.a. thermodynamic properties): heat capacities; equations of state relating pressure, density and temperature; etc. Other properties of interest include some dynamical information. This is the case for transport coefficients (such as thermal conductivity, shear viscosity, etc) or time-dependent dynamic properties such as Arrhenius constants which parametrize chemical kinetics.

From a technical viewpoint, the computation of macroscopic properties requires in any case the sampling of high-dimensional measures. We consider in this work the computation of properties in the canonical ensemble, characterized by the Boltzmann--Gibbs measure, which models systems at constant temperature. One popular way to sample the canonical ensemble is provided by the Langevin dynamics. Denoting by $D$ the dimension of the system, by $q \in \calD$ the positions of the particles in the system and by $p \in \bbR^D$ their momenta, the Langevin dynamics reads
\begin{equation}
\label{eq:langevin dynamics}
\left\{
\begin{aligned}
  \dd q_t &= \frac{p_t}{m} \, \dd t, \\
  \dd p_t &= \left(-\nabla V(q_t) - \gamma \frac{p_t}{m} \right) \, \dd t + \sqrt{\frac{2 \gamma}{\beta}} \, \dd W_t,
\end{aligned}
\right.
\end{equation}
where $\beta>0$ is proportional to the inverse temperature, $m>0$ is the mass of the particles\footnote{Our results can be extended to the case of any symmetric positive definite mass matrix $M$ but we focus on the case when $M$ is proportional to the identity matrix for simplicity.}, $\gamma > 0$ is the friction coefficient and $W_t$ is a standard Brownian motion in dimension~$D$. The potential energy $V:\calD \rightarrow \bbR$ is supposed to be a smooth function. In practice, $\calD$ is either a compact domain with periodic boundary conditions, as for example $\calD=(a\mathbb{T})^{D}$ where $\mathbb{T} = \bbR/\bbZ$ is the unit torus and $a>0$ denotes the size of the simulation cell; or the unbounded space $\calD = \bbR^D$.  
When $\rme^{-\beta V}$ is integrable, the Langevin dynamics admits as a unique invariant measure the canonical measure
\begin{equation}
  \label{eq:mueq}
  \mu(\dd q \, \dd p) = \Zinvmu \rme^{-\beta H(q,p)} \, \dd q \, \dd p, 
  \qquad  
  H(q,p) = V(q) + \frac {|p|^2} {2m},
\end{equation}
where the partition functions $Z_{\beta, \mu}$ is a normalization coefficient. 

In several situations, one is interested in solutions of Poisson equations of the form 
\begin{equation}
  \label{eq:poisson problem}
  -\calL \Phi = R - \Esp_\mu[R],
\end{equation}
where $\calL$ denotes the generator of the Langevin dynamics~\eqref{eq:langevin dynamics}. For instance, asymptotic variances of ergodic averages or transport coefficients can be written as
\begin{equation}
  \label{eq:autocorrelation}
  \int_\calE -\calLinv \left( R - \Esp_\mu[R] \right)  S \, \dd \mu
\end{equation}
for some functions $R$ and $S$. For the asymptotic variance related to the time average of an observable $R$, one has $S = 2R$. For transport coefficients, $R$ would be the system response whereas $S$ is the conjugate response (see for instance the presentation in~\cite[Section~5]{Lelievre16}). In practice, quantities such as~\eqref{eq:autocorrelation} are evaluated by Monte Carlo strategies, where the quantity of interest is rewritten as the integral of a time-dependent correlation function (the famous Green--Kubo formula), which is approximated by independent realizations of the process. In some cases however, spectral methods are used to solve the Poisson equation~\eqref{eq:poisson problem}, see for instance~\cite{Risken96,Latorre13,Redon16,PavVog08}.

The error analysis associated with spectral Galerkin methods faces several difficulties. The most important one probably is that the generator $\calL$ of the Langevin dynamics is not an elliptic operator, and that it is not naturally associated with a quadratic form. Many approximation results exist for elliptic operators, see for instance~\cite{Chatelin11}. In the context of molecular dynamics, elliptic operators correspond to overdamped Langevin dynamics, which are effective dynamics on the positions only. A Lax-Milgram theorem holds for the quadratic form associated with the generator of the overdamped Langevin dynamics, which makes it possible to quantify the error on the solution of Poisson equations, as recently done in~\cite{Abdulle17}. In contrast, the generator $\calL$ of the Langevin dynamics~\eqref{eq:langevin dynamics} is invertible but not coercive, so that a dedicated treatment is required to obtain error estimates. This is done here by a perturbation of the proof of invertibility obtained as a corollary of the decay estimates provided in~\cite{Dolbeault09,Dolbeault15}, which builds on the theory of hypocoercivity~\cite{Villani09}. Note that this proof applies to a large class of hypocoercive operators. In this work we restrict ourselves to the Langevin dynamics, the proofs being directly transposable for operators satisfying the hypotheses presented in~\cite{Dolbeault15}. Let us also mention previous results on the numerical analysis of hypocoercive operators, relying on finite element or finite difference methods, and providing finite time estimates~\cite{Foster14,Porretta17}.

\medskip

This article is organized as follows. We first recall some fundamental properties of the Langevin dynamics in Section~\ref{s:equili}, where we describe in particular the approach developed in~\cite{Dolbeault09,Dolbeault15}. We next provide in Section~\ref{s:discrete convergence} general a priori error estimates for the solutions of Poisson equations~\eqref{eq:poisson problem}. One of the key point to state such error estimates is to prove the invertibility of the generator restricted to the Galerkin space, which can be shown by adapting the hypocoercive approach of~\cite{Dolbeault09,Dolbeault15}. We finally turn in Section~\ref{s:eq appli} to an application to a simple, one-dimensional setting, where explicit convergence rates can be obtained. Numerical simulations are also performed to test the relevance of the bounds we provide. Some technical results are gathered in the appendices.

\section{Convergence of the Langevin dynamics}
\label{s:equili}

We recall in this section useful theoretical results on exponential convergence rates for the semigroup $\rme^{t \calL}$ associated with the generator of the Langevin dynamics, following the methodology introduced in~\cite{Dolbeault09,Dolbeault15} and further made precise in~\cite{Grothaus14} (note that the latter works rather considered the adjoint of the generator~$\calL$, the so-called Fokker--Planck operator, but this does not change the structure of the proof, see Remark~\ref{rmk:convergence_FP} below); see also~\cite{Iacobucci17} for an application to Langevin dynamics. We formulate the result both for bounded and unbounded position spaces.

In the following we consider all operators as defined on the Hilbert space $L^2(\mu)$. The adjoint of a closed operator $T$ on $\Lmu$ is denoted by $T^*$. The scalar product and norm on $\Lmu$ are respectively denoted by $\lang \cdot, \cdot \rang$ and $\| \cdot \|$. In fact, it is convenient in many cases to work in the subspace 
\begin{equation}
  \tLmu = \left\{ \varphi \in \Lmu \left| \int_\calE \varphi \, \dd \mu = 0 \right. \right\}
\end{equation}
of $L^2(\mu)$. The orthogonal projector onto $\tLmu$ is defined by
\begin{equation}
  \forall \varphi \in \Lmu, \qquad \Pinot \varphi = \varphi - \Espmu(\varphi).	
\end{equation}
Since 
\[
\left(\rme^{t \calL}\varphi\right)(q,p) = \mathbb{E}\left(\varphi(q_t,p_t) \, \Big| \, (q_0,p_0)=(q,p) \right)
\]
where the expectation is over all the realizations of the Brownian motion in~\eqref{eq:langevin dynamics}, it is expected that $\rme^{t \calL}\varphi$ converges to $\Espmu(\varphi)$. Therefore, $\rme^{t \calL}\varphi$ converges to~0 for $\varphi \in \tLmu$. In order to state a precise convergence result, we need some conditions on the potential~$V$, and on the marginal measure  of $\mu$ in the position variable. The marginal measures in the position and momentum variables are respectively
\begin{equation}
  \label{eq:marginal measures}
  \nu(\dd q) = \Zinvnu \rme^{-\beta V(q)} \, \dd q, 
  \qquad  
  \kappa(\dd p) = \left(\frac{\beta}{2\pi m}\right)^{D/2} \rme^{-\beta \frac {|p|^2} {2 m} } \, \, \dd p.
\end{equation}
We denote by $\Hsnu$ the weighted Sobolev spaces of index $s \in \bbN$ composed of functions $\varphi(q)$ of the position variables for which $\partial_q^\alpha \varphi \in \Lmu$ for any multi-index $\alpha=(\alpha_1,\dots,\alpha_D) \in \bbN^D$ such that $|\alpha| = \alpha_1 + \dots \alpha_D \leq s$ (where $\partial_q^\alpha = \partial_{q_1}^{\alpha_1} \dots \partial_{q_D}^{\alpha_D}$). The spaces $\Hskappa$ and $\Hsmu$ are defined in a similar way.

\begin{assumption}
  \label{ass:potential}
  The potential $V$ is smooth, and the marginal measure $\nu$ satisfies a Poincaré inequality with constant $\Cnu > 0$: for any function of the positions $\varphi \in \Hnu$,
  \begin{equation}
    \label{eq:poincare nu}
    \left\| \varphi - \int_\calD \varphi \, \dd \nu \right\|_\Lnu^2 \leq \frac 1 {\Cnu} \| \naq \varphi \|_\Lnu^2.
  \end{equation}
  Moreover, there exist $c_1 > 0$, $c_2 \in [0,1)$ and $c_3 > 0$ such that $V$ satisfies
    \begin{equation}
      \label{eq:regularization condition}
      \Delta V \leq c_1 + \frac {c_2} 2 | \nabla V |^2, \quad |\nabla^2 V | \leq c_3 \left( 1 + | \nabla V | \right).
    \end{equation}
\end{assumption}

Note that conditions~\eqref{eq:poincare nu} and~\eqref{eq:regularization condition} are automatically satisfied when $\calD$ is compact. The Poincar\'e inequality holds when there exists $a \in (0,1)$ such that (see~\cite{Bakry08})
\begin{equation}
  \liminf_{|q| \to \infty} a \beta | \nabla V (q)|^2 - \Delta V(q) > 0.
\end{equation}
The precise convergence result is then the following~\cite{Dolbeault09,Dolbeault15} (the proof is recalled in Appendix~\ref{app:proof hypocoercivity}).

\begin{theorem}[Hypocoercivity in $\Lmu$]
  \label{th:hypocoercivity}
  Suppose that Assumption~\ref{ass:potential} holds. Then there exist $C > 0$ and $\lambda_\gamma > 0$ (which are explicitly computable in terms of the parameters of the dynamics, $C$ being independent of $\gamma>0$) such that, for any initial datum $\varphi \in \tLmu$,
  \begin{equation}
    \label{eq:cv_expo_L2}
    \forall t \geq 0, \qquad \left\| \rme^{t \calL} \varphi \right\| \leq C \rme^{-\lambda_\gamma t} \| \varphi \|.
  \end{equation}
  Moreover, the convergence rate is of order $\min(\gamma,\gamma^{-1})$: there exists $\overline{\lambda} > 0$ such that
  \[
  \lambda_\gamma \geq \overline{\lambda} \min(\gamma,\gamma^{-1}).
  \]
\end{theorem}

\begin{remark}
  \label{rmk:convergence_FP}
  Theorem~\ref{th:hypocoercivity} admits a dual version in terms of probability measures. Consider an initial condition $\psi_0 \in \Lmu$, which represents the density with respect to $\mu$ of a probability measure $f_0 = \psi_0 \mu$. In particular, 
  \[
  \psi_0 \geq 0, \qquad \int_\calE \psi_0 \, d\mu = 1.
  \]
  Then the time-evolved probability measure $f_t = \psi_t \mu$ with $\psi_t = \rme^{t \calL^*} \psi_0$ converges exponentially fast to~$\mu$ in the following sense:
  \begin{equation}
    \forall t \geq 0, \qquad \| \psi_t - \bfone \| \leq C \rme^{-\lambda_\gamma t} \| \psi_0 \|.
  \end{equation}
\end{remark}

The convergence result~\eqref{eq:cv_expo_L2} can be used to deduce that $\calL$ is invertible on $\tLmu$. We denote by $\calB(E)$ the Banach space of bounded operators on a given Banach space $E$, endowed with the norm
\[
\| T \|_{\calB(E)} = \sup_{\varphi \in E\backslash \{0\}} \frac{\|T\varphi\|_E}{\|\varphi\|_E}.
\]
We simply denote by $\|T\|$ the operator norm on~$\Lmu$.
 
\begin{corollary}
  \label{coro:bounded inverse continuous}
  The operator $\calL$ is invertible on $\tLmu$, with
  \[
  \calLinv = -\int_0^\infty \rme^{t \calL} \, \dt 
  \qquad 
  \| \calLinv \|_{\calB(\tLmu)} \leq \frac C {\overline{\lambda}} \max(\gamma,\gamma^{-1}).
  \]
\end{corollary}

The upper bound on the resolvent is sharp in terms of the scaling with respect to~$\gamma$, as shown in~\cite{Hairer08} for $\gamma\to 0$ and~\cite{Leimkuhler16} for $\gamma \to +\infty$; see also~\cite{Kozlov89} for the case $V=0$. 

In particular, the Poisson problem~\eqref{eq:poisson problem} admits a unique solution $\Phi \in \tLmu$ for any observable $R \in \Lmu$. In order to capture the solution $\Phi$ of~\eqref{eq:poisson problem} numerically, one possibility is to discretize the operator $\calL$ on a Galerkin subspace of $\tLmu$. Section~\ref{s:discrete convergence} proves the convergence of this method under appropriate assumptions.

\medskip

Let us conclude this section by highlighting some elements of the proof of Theorem~\ref{th:hypocoercivity}, which will be needed to establish a convergence result similar to~\eqref{eq:cv_expo_L2} when a Galerkin discretization is considered.  In order to formulate the result more rigorously, we introduce the core $R$ composed of all $\calC^\infty$ functions with compact support. The first key element in the proof is to use a modified norm equivalent to the standard $\Lmu$ norm. To define this norm, the generator $\calL$ is decomposed into a symmetric part (corresponding to the fluctuation/dissipation) and an anti-symmetric part (corresponding to Hamiltonian transport):
\begin{equation}
\label{eq:generator}
\calL = \Lham + \gamma \LFD , \qquad \mbox{with} \qquad
\left\{
\begin{aligned}
  \Lham &= \left(\frac{p}{m}\right)^\top \naq - \nabla V^\top \nabla_p, \\
  \LFD &= -\left(\frac{p}{m}\right)^\top \nabla_p + \frac1\beta \Delta_p.
\end{aligned}
\right.
\end{equation}
With this notation, $\Lham^* = -\Lham$ while $\calL_{\rm FD}^* = \LFD$. In fact, since 
\[
\nabla_p^* = -\nabla_p^\top + \beta \frac{p^\top}{m}, \qquad \nabla_q^* = -\nabla_q^\top + \beta \nabla V^\top,
\]
the two parts of the generator $\calL$ can be reformulated as
\begin{equation}
\label{eq:reformulate Leq}
  \LFD = -\frac1\beta \nap^* \nap, 
  \qquad
  \Lham = \frac{1}{\beta}\left( \nap^* \naq - \naq^* \nap \right).
\end{equation}
We also need the orthogonal projector in $\tLmu$ on the subspace of functions depending only on positions:
\begin{equation}
  \label{eq:def_Pip}
  \forall \varphi \in \Lmu, \qquad \left(\Pip \varphi\right)(q) = \int_{\bbR^D} \varphi(q,p) \, \kappa(\dd p).
\end{equation}

\begin{definition}[Modified squared $\Lmu$ norm]
  Fix $\varepsilon \in (-1,1)$. For any function $\varphi \in \core$, 
  \begin{equation}
    \label{eq:entropy_functional}
    \calH[\varphi] = \frac 1 2 \|\varphi\|^2 - \varepsilon \lang A \varphi, \varphi \rang,
    \qquad
    A = \Big( 1 + (\Lham \Pip)^* (\Lham \Pip) \Big)^{-1} (\Lham \Pip)^*.
  \end{equation}	
\end{definition}

A more explicit expression of the operator~$A$ is provided in~\eqref{eq:reformulate A}. Since this operator is used in the sequel to state some conditions required for the error estimates, we gather some of its properties in the following lemma.

\begin{lemma}
  \label{lemma:equivalent norms}
  It holds $A = \Pip A (1-\Pip)$. Moreover, for any $\varphi \in \Lmu$,
  \[
  \| A \varphi \| \leq \frac 1 2 \|(1-\Pip)\varphi \|, 
  \qquad 
  \| \Lham A \varphi \| \leq \|(1-\Pip)\varphi \|.
  \]
\end{lemma}

In particular, the operator $A$ is in fact bounded in $\Lmu$ with operator norm smaller than~1, so that $\sqrt{\calH}$ is a norm equivalent to the canonical norm of $\Lmu$ for $-1 < \varepsilon < 1$:
\begin{equation}
  \label{eq:equivalent norms}
  \frac {1-\varepsilon} 2  \| \varphi \|^2 \leq \calH[\varphi] \leq \frac {1+\varepsilon} 2 \| \varphi \|^2.
\end{equation}

The second key element is a coercivity property enjoyed by the time-derivative of the entropy functional. Denoting by $\lang \lang \cdot, \cdot \rang \rang$ the scalar product associated by polarization with $\calH$, the following result can be proved.

\begin{prop}
  \label{prop:coercivity_scrD}
  There exists $\overline{\varepsilon} \in (0,1)$ and $\overline{\lambda}>0$, such that, by considering $\varepsilon = \overline{\varepsilon} \min(\gamma,\gamma^{-1})$ in~\eqref{eq:entropy_functional},
  \begin{equation}
    \label{eq:coercivity double angle}
    \forall \varphi \in \Pinot \core, \qquad \scrD[\varphi] := \lang \lang -\calL \varphi, \varphi \rang \rang \geq \widetilde{\lambda}_\gamma \|\varphi\|^2,
  \end{equation}
  with $\widetilde{\lambda}_\gamma \geq \overline{\lambda} \min(\gamma,\gamma^{-1})$.
\end{prop}

This coercivity property and a Gronwall inequality then allow to conclude to the exponential convergence to~0 of $\calH[\rme^{t \calL}\varphi]$, for any smooth function $\varphi$ with zero mean. Equation~\eqref{eq:cv_expo_L2} follows by the norm equivalence of $\sqrt{\calH}$ and $\| \cdot \|$.

\section{General a priori error estimates}
\label{s:discrete convergence}

In order to approximate the solution of the Poisson equation~\eqref{eq:poisson problem}, we consider a Galerkin discretization characterized by a finite dimensional subspace $\VM \subset \Lmu$. We present the structure of the proof of error estimates in the conformal case (\emph{i.e.} $\VM \subset \tLmu$) for the sake of clarity. Results in the non-conformal case are presented later on. Note that the results presented in this section for the Langevin generator can be generalized to other hypocoercive generators satisfying the assumptions required in~\cite{Dolbeault09, Dolbeault15}. For conformal discretization spaces, the approximate solution $\PhiM$ is defined by the variational formulation
\begin{equation}
  \label{eq:weak formulation}
  \left\{ \begin{aligned}
    &\mbox{Find} \ \PhiM \in \VM \ \mbox{such that} \\
    &\forall \psi \in \VM, \ -\lang \psi, \calL \PhiM \rang = \lang \psi, R \rang.
  \end{aligned} \right.	
\end{equation}
Note that $\Pinot R$ can be replaced by $R$ on the right-hand side since functions $\psi \in \VM$ have average~0 with respect to~$\mu$. Denoting by $\PiM$ the projector onto $\VM$, the variational formulation can be rewritten as 
\[
-\PiM \calL \PiM \PhiM = \PiM R.
\]
We first prove in this section the existence and uniqueness of the solution $\PhiM$ of~\eqref{eq:weak formulation} by studying the discretized operator $-\PiM \calL \PiM$. A dedicated study is required since the generator $\calL$ is invertible but not coercive on $\tLmu$, so that the Lax-Milgram theorem cannot be applied. This is a major difference with overdamped Langevin dynamics for which the discretized problem is automatically well posed when a Poincar\'e inequality holds true~\cite{Abdulle17}. Note that there are scalar products for which the quadratic form induced by $-\calL$ is coercive, for instance the one induced by polarization from~$\calH$ or the scalar product on~$\rmH^1(\mu)$ introduced in the hypocoercivity setting considered in~\cite{Hairer08,Villani09}. These scalar products however depend on parameters which are not explicitly known and on the friction~$\gamma$, so that they cannot be considered for numerical simulations.

We study instead the existence and the uniqueness of the solution $\PhiM$ by a perturbation of the proof of Theorem~\ref{th:hypocoercivity}, in two settings: the conformal case $\VM \subset \tLmu$ (see Subsection~\ref{ss:conformal}) and the non-conformal case $\VM \subset \Lmu$ but $\VM \not\subset \tLmu$ (the functions in the Galerkin basis are not of mean~0 with respect to~$\mu$, see Subsection~\ref{ss:non-conformal}). 

In a second step, we prove a priori error estimates. To this end, we decompose the difference between $\PhiM$ and the solution $\Phi$ of the equation~\eqref{eq:poisson problem} as the sum of two terms:
\begin{equation}
  \label{eq:decomposition of the error}
  \PhiM - \Phi = (\PhiM - \PiM \Phi) - (1-\PiM) \Phi.
\end{equation}
The second term on the right-hand side is the approximation error $(1-\PiM) \Phi$, which depends only on the Galerkin space. We therefore postpone the study of this error to specific models (see Section~\ref{ss:approximation error}). The first term is related to the consistency error $\eta_M = \PiM \calL \PiM \Phi + \PiM R$ since $\PhiM - \PiM \Phi = \left(-\PiM \calL \PiM\right)^{-1} \eta_M$. We provide general error estimates on $\PhiM - \PiM \Phi$ in Section~\ref{ss:consistency}. They can be made more precise in specific contexts, with explicit convergence rates; see Section~\ref{sec:consistency_specific_model}.

We conclude the section with a practical reformulation of the variational problem~\eqref{eq:weak formulation} in a form more amenable to numerical computations (see Section~\ref{ss:matrices}).

\subsection{Conformal case}
\label{ss:conformal}

In this section we suppose that $\VM \subset \tLmu$. The following theorem states that if the additional terms arising from the discretization in the expression of the entropy dissipation are sufficiently small, then hypocoercivity holds on the subspace $\VM$, and the exponential rate of convergence to~0 of the semigroup associated with $\PiM \calL \PiM$ is uniform in $M$.

\begin{theorem}[Discrete hypocoercivity]
  \label{th:eq discrete hypocoercivity}
  Fix $\gamma > 0$. Assume that the Galerkin space is composed of functions with mean~0 with respect to~$\mu$ (\ie $\VM \subset \tLmu$) and that
  \begin{equation}
    \label{eq:conditions hypocoercivity}
    \| (A+A^*) (1-\PiM) \calL \PiM \| \xrightarrow[M \to \infty]{} 0.
  \end{equation}
  Then there exist $C \geq 1$ (independent of $M,\gamma$) and $M_0 \in \bbN$ such that, for any $M \geq M_0$, there is $\lambdagammaM > 0$ for which 
  \begin{equation}
    \forall \varphi \in \VM, \quad \forall t \geq 0, 
    \qquad 
    \left\| \rme^{t \PiM \calL \PiM} \varphi \right\| \leq C \rme^{-\lambdagammaM t} \| \varphi \|.
  \end{equation}
  Moreover, $\lambdagammaM \xrightarrow[M \to \infty]{} \lambdagamma$ where $\lambdagamma > 0$ is introduced in~\eqref{eq:cv_expo_L2}.

  If in addition $\LFD$ stabilizes $\VM$ (in the sense that $\PiM \LFD = \LFD \PiM$), then there exist $M_* \geq 1$ (independent of $\gamma$) such that, for any $M \geq M_*$, the following uniform bound holds: 
  \begin{equation}
    \label{eq:estimate_lambda_gamma_M_unif}
    \forall \gamma > 0, \qquad \lambdagammaM \geq \overline{\lambda}_M \min(\gamma,\gamma^{-1}),
  \end{equation}
  with $\overline{\lambda}_M  \xrightarrow[M \to \infty]{} \overline{\lambda}$ where $\overline{\lambda} > 0$ is introduced in Proposition~\ref{prop:coercivity_scrD}.
\end{theorem}

Let us emphasize that the condition~\eqref{eq:conditions hypocoercivity} should be checked for the specific model under consideration; see Appendix~\ref{app:proof discrete hypocoercivity} for an example. Note that the left hand side of~\eqref{eq:conditions hypocoercivity} is constituted of a regularization operator $A+A^*$ applied to a residual off diagonal part of the operator $\calL$. It is therefore expected that the norm of this operator goes to zero.

The stability of $V_M$ by $\LFD$ is automatically ensured when the basis functions are tensor products of functions of the positions and eigenfunctions of $\LFD$ for the momentum part. The latter eigenfunctions turn out to be analytically known (they are in fact appropriately scaled Hermite functions, see Section~\ref{ss:system and galerkin}), which makes it easy to conclude to~\eqref{eq:estimate_lambda_gamma_M_unif}.

\begin{proof}
Fix $\varphi_0 \in \VM$ and $\gamma > 0$, and consider $\varepsilon = \overline{\varepsilon}\min(\gamma,\gamma^{-1})$ as in Proposition~\ref{prop:coercivity_scrD}. Introduce $\varphi_M(t) = \exp(t \PiM \calL \PiM) \varphi_0$ and $\scrH_M(t) = \calH[\varphi_M(t)]$. Note that the discretized generator $\PiM \calL \PiM$ stabilizes the Galerkin space $\VM \subset \tLmu$. In particular, $\varphi_M(t) \in \VM \subset \tLmu$ for all $t \geq 0$ when $\varphi_0 \in \VM$. The time-derivative of the entropy functional is $\scrH'_M(t) = - \scrD_M[\varphi_M(t)]$, where $\scrD_M$ is similar to the entropy dissipation defined in~\eqref{eq:coercivity double angle} apart from two additional terms arising from the discretization. More precisely, for $\varphi \in V_M$, 
\begin{equation}
  \label{eq:DM ineqality conformal}
  \begin{aligned}
    \scrD_M[\varphi] &= -\lang \varphi , \PiM \calL \PiM \varphi \rang 
    - \varepsilon \lang A \PiM \calL \PiM \varphi , \varphi \rang - \varepsilon \lang A \varphi , \PiM \calL \PiM \varphi \rang \\
    &= -\lang \varphi , \calL \varphi \rang 
    - \varepsilon \lang A \PiM \calL \varphi , \varphi \rang - \varepsilon \lang \varphi , A^* \PiM \calL \varphi \rang \\
    &= \scrD[\varphi] + \varepsilon \lang A (1-\PiM) \calL \varphi, \varphi \rang + \varepsilon \lang \varphi , A^* (1-\PiM) \calL  \varphi \rang \\
    &\geq \scrD[\varphi] - \varepsilon \| (A+A^*) (1-\PiM) \calL \PiM \varphi\| \, \| \varphi \| \\
 &\geq \Big( \widetilde{\lambda}_\gamma - \varepsilon \| (A+A^*) (1-\PiM) \calL \PiM \| \Big) \| \varphi \|^2,
  \end{aligned}
\end{equation}
where the last inequality follows from Proposition~\ref{prop:coercivity_scrD}. The conclusion then follows from the same reasoning as the one used at the end of Appendix~\ref{app:proof hypocoercivity} to prove Theorem~\ref{th:hypocoercivity}, with an exponential convergence rate which is degraded uniformly in $M$:
\begin{equation}
  \label{eq:estimate_lambda_gamma_M}
  \lambdagammaM = \lambdagamma - \frac {\varepsilon}{1+\varepsilon} \| (A+A^*) (1-\PiM) \calL \PiM \| > 0
\end{equation}
for $M$ large enough.

Assume now that 
\begin{equation}
  \label{eq:commutation LFD PiM}
  \LFD \PiM = \PiM \LFD
\end{equation}
so that $(1-\PiM) \calL \PiM = (1-\PiM) \Lham \PiM$ does not depend on~$\gamma$. The only $\gamma$-dependence on the right-hand side of~\eqref{eq:estimate_lambda_gamma_M} therefore arises from $\varepsilon = \overline{\varepsilon}\min(\gamma,\gamma^{-1})$. We then deduce the following lower bound from~\eqref{prop:coercivity_scrD}:
\begin{equation*}
\lambdagammaM \geq \left( \overline{\lambda} - \overline{\varepsilon} \| (A+A^*) (1-\PiM) \Lham \PiM \| \right) \min(\gamma,\gamma^{-1}),
\end{equation*}
which implies~\eqref{eq:estimate_lambda_gamma_M_unif}.
\end{proof}

\begin{remark}
Another way to prove the hypocoercivity of the discretized generator on~$\Lmu$ would be to first prove this property on~$\Hmu$ (as in~\cite{Villani09}), and then use hypoelliptic regularization~\cite{Herau07}. This program is performed for Langevin dynamics in~\cite{Hairer08}, with an emphasis on the Hamiltonian limit $\gamma \to 0$  (see also~\cite[Sections~2.3.3 and~2.3.4]{Lelievre16} for a careful analysis of the two limiting regimes $\gamma\to 0$ and $\gamma \to +\infty$). This approach introduces scalar products on~$\Hmu$ depending on three coefficients $a,b,c \in \bbR$. The corresponding proofs are therefore more involved than the approach described here, and, more importantly, the conditions for $\Hmu$ hypocoercivity are incompatible with the conditions for $\Lmu$ regularization for the Galerkin space proposed in Section~\ref{s:eq appli}; see~\cite{Roussel18} for further precisions. 
\end{remark}

An immediate consequence of the convergence result stated in Theorem~\ref{th:eq discrete hypocoercivity} is the following corollary. It states that the discrete operator has a spectral gap, which does not vanish when the size of the Galerkin basis increases.

\begin{corollary}[Discrete invertibility]
  \label{coro:bounded inverse discrete}
  For any $M \geq M_0$, the operator $\PiM \calL \PiM$ is invertible on~$\VM$ and the following equality holds on~$\calB(\VM)$:
  \[
  (\PiM \calL \PiM)^{-1} = -\int_0^\infty \rme^{t \PiM \calL \PiM} \dt.
  \]
  Moreover,
  \[
  \left\| (\PiM \calL \PiM)^{-1} \right\|_{\calB(\VM)} \leq \frac C {\lambdagammaM}.
  \]
\end{corollary}

In particular, when $\LFD$ stabilizes $V_M$, the dependence on~$\gamma$ of the resolvent bound can be made explicit thanks to~\eqref{eq:estimate_lambda_gamma_M_unif}. Corollary~\ref{coro:bounded inverse discrete} shows that the Galerkin problem~\eqref{eq:weak formulation} admits a unique solution, denoted by $\PhiM = -\left( \PiM \calL \PiM \right)^{-1} \PiM R$.

\subsection{Non-conformal case}
\label{ss:non-conformal}

In practice the assumption $\VM \subset \tLmu$ is constraining since it may not be convenient to construct a basis of $\tLmu$ which is orthogonal for the associated scalar product. It seems easier in many situations to consider bases which are orthonormal on~$\Lmu$ rather than $\tLmu$ (as we do here for the application treated in Section~\ref{s:eq appli}). Moreover, it may be preferable in practice to create bases adapted to the operators $\naq$, $\naq^*$, $\nap$ and $\naq^*$ in order to simplify the algebra involved in the computation of the elements of the rigidity matrix. For these two reasons basis functions are rarely of mean~0 with respect to~$\mu$ in the literature, see for instance~\cite{Latorre13, Redon16, Abdulle17} for recent examples. We therefore need to extend the results of Section~\ref{s:discrete convergence} to the non-conformal case $V_M \not \subset \tLmu$. 

Now, the generator $\calL$ is invertible on $\tLmu$ (by Corollary~\ref{coro:bounded inverse continuous}) but not on $\Lmu$ since $\calL \bfone = 0$. The purpose of this subsection is to show how this degeneracy can be dealt with by introducing a Lagrangian formulation. We start by applying Theorem~\ref{th:eq discrete hypocoercivity} to the Galerkin space $\Vcap = \VM \cap \tLmu$, whose associated orthogonal projector we denote by~$\Picap$. The issue is to control the solution in the direction associated with the function
\begin{equation}
 u_M = \frac {\PiM \bfone} {\| \PiM \bfone \|} \in \VM,
\end{equation}
which is not of zero mean. In this setting the approximate solution $\PhiM$ is defined by the variational formulation
\begin{equation}
  \label{eq:weak formulation non conformal}
  \left\{ \begin{aligned}
    &\mbox{Find} \ \PhiM \in \Vcap \ \mbox{such that} \\
    &\forall \psi \in \Vcap, \ -\lang \psi, \calL \PhiM \rang = \lang \psi, R \rang,
  \end{aligned} \right.	
\end{equation}
which can be rewritten as 
\[
-\Picap \calL \Picap \PhiM = \Picap R.
\]
The precise result is the following.

\begin{corollary}[Non-conformal Galerkin method]
  \label{coro:non-conformal hypocoercivity}
  Assume that the Galerkin space~$\VM$ is such that~\eqref{eq:conditions hypocoercivity} holds and additionally that
  \begin{equation}
    \label{eq:additional conditions non-conformal}
    \quad \| \calL^* u_M \| \xrightarrow[M \to \infty]{} 0.
  \end{equation}
  Then there exist $C \geq 1$ (independent of $M,\gamma$) and $M_0 \geq 1$ such that, for any $M \geq M_0$, the operator $\Picap \calL \Picap$ is invertible on $\VM$ and there is $\tlambdagammaM > 0$ for which 
  \[
  \left\| \left( \Picap \calL \Picap \right)^{-1} \right\|_{\calB(\Vcap)} \leq \frac C {\tlambdagammaM},
  \]
  with $\tlambdagammaM \xrightarrow[M \to \infty]{} \lambdagamma > 0$ where $\lambdagamma > 0$ is introduced in~\eqref{eq:cv_expo_L2}.

  If in addition $\LFD$ stabilizes $\VM$, then there exist $M_* \geq 1$ (independent of $\gamma$) such that, for any $M \geq M_*$, the following uniform bound holds: 
  \[
  \forall \gamma > 0, \qquad \tlambdagammaM \geq \overline{\lambda}_M \min(\gamma,\gamma^{-1}),
  \]
  with $\overline{\lambda}_M  \xrightarrow[M \to \infty]{} \overline{\lambda}$ where $\overline{\lambda} > 0$ is introduced in Proposition~\ref{prop:coercivity_scrD}.
\end{corollary}

\begin{proof}
Let us first decompose $\VM$ as an orthogonal direct sum:
\[
\VM = \Vcap  \oplus \bbR u_M.
\]
Denoting by $\Pi_{u_M}$ the orthogonal projection onto $\bbR u_M$, it then holds $\PiM = \Picap + \Pi_{u_M}$. We can now show how the hypotheses on $\PiM$ allow to apply Theorem~\ref{th:eq discrete hypocoercivity} on the Galerkin space~$\Vcap$. We follow the proof of Theorem~\ref{th:eq discrete hypocoercivity} until~\eqref{eq:DM ineqality conformal}, replacing $\PiM$ with $\Picap$. It then suffices to prove that the following term is of order~$\| \varphi \|^2$ for any $\varphi \in \Vcap$:
\[
\lang (A+A^*) (1-\Picap) \calL \varphi, \varphi \rang = \lang (A+A^*) (1-\PiM) \calL \varphi, \varphi \rang + \lang (A+A^*) \Pi_{u_M} \calL \varphi, \varphi \rang.
\]
The first term on the right-hand side can be dealt with as in the proof of Theorem~\ref{th:eq discrete hypocoercivity}, making use of~\eqref{eq:conditions hypocoercivity}. For the second one, we remark that
\[
\lang (A+A^*) \Pi_{u_M} \calL \varphi, \varphi \rang = \lang \calL \varphi, u_M \rang \lang (A+A^*) u_M, \varphi \rang,
\]
so that, using $\| A \| = \| A^* \| \leq 1/2$ (from Lemma~\ref{lemma:equivalent norms}):
\begin{equation}
	\left| \lang (A+A^*) \Pi_{u_M} \calL \varphi, \varphi \rang \right| \leq \| \varphi \| \, \| \calL^* u_M \| \, \left\| (A+A^*) u_M \right\| \, \| \varphi \| \leq \| \calL^* u_M \| \, \| \varphi \|^2.
\end{equation}
Plugging this additional term into the bound~\eqref{eq:DM ineqality conformal} obtained in the conformal case, it follows
\[
\scrD_M[\varphi] \geq \left( \widetilde{\lambda}_\gamma - \varepsilon \left\| (A+A^*) (1-\PiM) \calL \PiM \right\| - \varepsilon \| \calL^* u_M \| \right) \| \varphi \|^2.
\]
We can then conclude to the exponential convergence of the semi-group, with rate
\begin{equation}
  \label{eq:estimate_tlambdagammaM}
  \tlambdagammaM = \lambdagamma - \frac {\varepsilon}{1+\varepsilon} \Big( \| (A+A^*) (1-\PiM) \calL \PiM \| + \| \calL^* u_M \| \Big) > 0,
\end{equation}
when $M$ is sufficiently large. The remainder of the proof follows the lines of the end of the proof of Theorem~\ref{th:eq discrete hypocoercivity}.
\end{proof}

Corollary~\ref{coro:non-conformal hypocoercivity} implies that the following saddle-point formulation is well-posed.

\begin{prop}[Saddle-point formulation]
  \label{prop:saddle point}
  Assume that \eqref{eq:conditions hypocoercivity} and \eqref{eq:additional conditions non-conformal} hold. Then, for any $R \in \Lmu$, there exist a unique $\PhiM \in \VM$ and a unique $\alpha_M \in \bbR$ such that
  \begin{equation}
    \label{eq:saddle point problem}
    \left\{ \begin{aligned}
      -\PiM \calL \PiM \PhiM + \alpha_M u_M &= \PiM R, \\
	\lang \PhiM, u_M \rang &= 0.
    \end{aligned} \right.
  \end{equation}
\end{prop}

Note that the unique solution $\PhiM$ in fact belongs to $\Vcap$ since $\lang \PhiM, u_M \rang = 0$. Moreover, $R$ does not need to be of mean~0 with respect to~$\mu$ thanks to the term $\alpha_M u_M$ on the left-hand side of the first equality in~\eqref{eq:saddle point problem}. We show in the next subsection that $\PhiM$ actually converges to the solution of the Poisson equation~\eqref{eq:poisson problem} with right-hand side~$\Pinot R$.

\begin{proof}
Consider $R \in \Lmu$. In view of Corollary~\ref{coro:non-conformal hypocoercivity}, there exists a unique $\PhiM \in \Vcap$ such that
\[
-\Picap \calL \PhiM = \Picap R.
\]
Recalling that $\Picap = \PiM - \Pi_{u_M}$ it follows that
\[
-\PiM \calL \PiM \PhiM + \Pi_{u_M} \left(\calL \PiM \PhiM + R \right) = \PiM R,
\]
which leads to the saddle-point formulation~\eqref{eq:saddle point problem} upon introducing the Lagrange multiplier $\alpha_M = \lang u_M, \calL \PiM \PhiM + R \rang$ (which is uniquely defined).
\end{proof}

The system~\eqref{eq:saddle point problem} can be reformulated as
\begin{equation}
  \label{eq:lagrangian formulation}
  \tcalLM \begin{pmatrix} \PhiM \\ \alpha_M \end{pmatrix} = \begin{pmatrix} \PiM R \\ 0 \end{pmatrix},
\end{equation}
where the Lagrangian operator $\tcalLM$ on $\VM \times \bbR$ reads
\begin{equation}
  \tcalLM \begin{pmatrix} \varphi \\ \alpha \end{pmatrix} = \begin{pmatrix} -\PiM \calL \PiM \varphi + \alpha u_M \\ \lang \varphi, u_M \rang \end{pmatrix}.
\end{equation}
Let us conclude this section by providing an estimate on the resolvent bound of $\tcalLM$. This estimate is used in Section~\ref{ss:matrices} to show that the matrix reformulation of~\eqref{eq:saddle point problem} is well-posed, and in fact enjoys a good conditioning. Let us first prove that the Lagrangian operator~$\tcalLM$ is invertible on~$\VM \times \bbR$ for $M \geq M_0$ (with $M_0$ the integer considered in Corollary~\ref{coro:non-conformal hypocoercivity}). This is done by proving that the equation $\tcalLM(\varphi,\alpha) = (\psi,s)$ admits a unique solution for an arbitrary element~$(\psi,s) \in \VM \times \bbR$. Note that 
\begin{equation}
  \label{eq:Lagrangian_well_posed}
  \tcalLM \begin{pmatrix} \varphi \\ \alpha \end{pmatrix} = \begin{pmatrix} \psi \\ s \end{pmatrix}  
\end{equation}
is equivalent to
\[
\tcalLM \begin{pmatrix} \varphi - s \, u_M \\ \alpha \end{pmatrix} = \begin{pmatrix} \psi - s \PiM \calL u_M \\ 0 \end{pmatrix}.
\]
For the latter equality to hold true, the function $\phi_{s,M} = \varphi-s u_M$ must satisfy the Poisson equation 
\begin{equation}
  \label{eq:Poisson_inv_Lagrangian}
  -\PiM \calL \PiM \phi_{s,M} = \psi - s \PiM \calL u_M - \alpha u_M, 
  \qquad
  \lang \phi_{s,M}, u_M \rang = 0.
\end{equation}
Then, $\phi_{s,M} \in \Vcap$ so that $\PiM \calL \PiM \phi_{s,M} = \PiM \calL \Picap \phi_{s,M} = \Picap  \calL \Picap \phi_{s,M} + \lang \calL \phi_{s,M}, u_M \rang u_M$. Therefore, \eqref{eq:Poisson_inv_Lagrangian} can be reformulated as 
\[
-\Picap \calL \Picap \phi_{s,M} = \psi - s \PiM \calL u_M + (\lang \calL \phi_{s,M}, u_M \rang u_M - \alpha) u_M, 
\qquad
\lang \phi_{s,M}, u_M \rang = 0.
\]
Since $\Picap \calL \Picap$ is invertible on $\Vcap$, the equation~\eqref{eq:Poisson_inv_Lagrangian} admits a unique solution in $\Vcap$ if and only if the right-hand side of the above Poisson equation is in $\Vcap$, which is the case if and only if
\begin{equation}
  \label{eq:alpha}
  \alpha = \lang u_M, \calL \PiM (\varphi-s u_M) + (\psi - s \PiM \calL u_M) \rang.
\end{equation}
This proves the existence and uniqueness of the solution to~\eqref{eq:Lagrangian_well_posed} since $\alpha$ and $\phi_{s,M}$ are completely identified through~\eqref{eq:alpha} and
\begin{equation}
  \label{eq:varphi}
  \varphi = su_M + \left(-\Picap \calL \Picap\right)^{-1} \Picap\left( \psi - s \PiM \calL u_M\right).
\end{equation}
This allows to conclude that $\tcalLM$ is invertible on $\VM \times \bbR$. Moreover, using Corollary~\ref{coro:non-conformal hypocoercivity},
\[
\| \varphi \|^2 \leq s^2 + \left(\frac C {\tlambdagammaM}\right)^2 \left( \| \psi \| + \| \calL u_M \| \, |s| \right)^2,
\]
and, in view of~\eqref{eq:alpha}-\eqref{eq:varphi},
\begin{equation}
  | \alpha | \leq \left( 1 + \| \calL^* u_M \| \frac C {\tlambdagammaM} \right) \left( \| \psi \| +\| \calL u_M \| \, |s| \right).
\end{equation}
Therefore, endowing $\VM \times \bbR$ with the norm associated with the canonical scalar product, the following resolvent bound holds:
\begin{equation}
  \label{eq:bound inverse tcalL}
  \left\| \tcalLM^{-1} \right\|_{\calB(\VM \times \bbR)}^2 \leq 1 + \left[ \left( \frac C {\tlambdagammaM} \right)^2 + \left( 1 + \frac C {\tlambdagammaM} \| \calL^* u_M \| \right)^2 \right] (1 + \| \calL u_M \|^2).
\end{equation}
In fact, the operators $\tcalLM^{-1}$ are bounded uniformly in~$M \geq M_0$, since the upper bound on $\left\| \tcalLM^{-1} \right\|_{\calB(\VM \times \bbR)}$ tends to $\sqrt{2 + (C/\lambdagamma)^2}$ as $M \to +\infty$.

\subsection{Consistency error}
\label{ss:consistency}

We study in this section the error $\| \PhiM - \Picap \Phi \|$ associated to the consistency error $\eta_{M,0} = \Picap \calL \Phi_M + \Picap R$, sticking to the non-conformal case since this setting is the most appropriate for actual applications. With some abuse of terminology, we simply call $\| \PhiM - \Picap \Phi \|$ the consistency error.

As in~\eqref{eq:decomposition of the error}, the error can be decomposed as 
\begin{equation}
  \label{eq:decomposition of the error non conformal}
  \PhiM - \Phi = (\PhiM - \Picap \Phi) - (1-\Picap) \Phi.
\end{equation}
Very similar results are obtained in the conformal case upon replacing $\Picap$ with $\PiM$. Moreover, we do not suppose in this section that $R$ has mean~0 with respect to~$\mu$, but consider the Poisson problem~\eqref{eq:poisson problem} with $R$ replaced by $\Pinot R$:
\begin{equation}
  \label{eq:poisson_problem_0}
  -\calL \Phi = \Pinot R. 
\end{equation}
The solution $\Phi$ is approximated by the solution of the Poisson equation
\begin{equation}
  \label{eq:poisson_problem_0_M}
  - \Picap \calL \Picap \Phi = \Picap R. 
\end{equation}
which is well-posed in view of Corollary~\ref{coro:non-conformal hypocoercivity}.

\begin{theorem}
  \label{th:consistency error}
  Assume that~\eqref{eq:conditions hypocoercivity} and~\eqref{eq:additional conditions non-conformal} hold. Then the consistency error between the unique solution $\Phi \in \tLmu$ of~\eqref{eq:poisson_problem_0} and the approximate solution $\Phi_M \in \Vcap$ of~\eqref{eq:poisson_problem_0_M} can be bounded by 
  \begin{equation}
    \label{eq:consistency_error_estimate}
    \| \PhiM - \Picap \Phi \| \leq \frac C {\tlambdagammaM} \left( \| \PiM \calL (1-\PiM) \Phi \| + \| \calL u_M \|\|\Phi\|\right),
  \end{equation}
  where $C,\tlambdagammaM$ are the constants introduced in Corollary~\ref{coro:non-conformal hypocoercivity}.
\end{theorem}

The extra term $\| \calL u_M \|\|\Phi\|$ on the right-hand side of~\eqref{eq:consistency_error_estimate} arises from the fact that the Galerkin space is not conformal. It would not be present for conformal spaces.

\begin{proof}
  Upon applying $\Picap$ to both sides of~\eqref{eq:poisson_problem_0}, it holds
  \[
  -\Picap \calL \Picap \Phi = \Picap R + \Picap \calL (1-\Picap) \Phi.
  \]
  After subtraction with~\eqref{eq:poisson_problem_0_M}, it follows
  \begin{equation}
    \label{eq:relation approx consist errors}
  \Picap \calL \Picap (\PhiM - \Picap \Phi) = \Picap \calL (1-\Picap) \Phi.
  \end{equation}
  Therefore, using Corollary~\ref{coro:non-conformal hypocoercivity},
  \begin{equation}
    \begin{aligned}
      \| \PhiM - \Picap \Phi \|_\Lmu &= \left\| (\Picap \calL \Picap)^{-1} \Picap \calL (1-\Picap) \Phi \right\|_\Lmu \\
      &\leq \left\| (\Picap \calL \Picap)^{-1} \right\|_{\calB(\Vcap)} \left\| \Picap \calL (1-\Picap) \Phi \right\|_\Lmu \\
      &\leq \frac C {\tlambdagammaM} \left\| \Picap \calL (1-\Picap) \Phi \right\|_\Lmu.
    \end{aligned}
  \end{equation}
  Moreover
  \begin{equation}
  \begin{aligned}
	  \left\| \Picap \calL (1-\Picap) \Phi \right\|_\Lmu &\leq \left\| \PiM \calL (1-\PiM+\Pi_{u_M}) \Phi \right\|_\Lmu \\
	  &\leq \left\| \PiM \calL (1-\PiM) \Phi \right\|_\Lmu + \left\| \PiM \calL \Pi_{u_M} \Phi \right\|_\Lmu \\
  	  &\leq \left\| \PiM \calL (1-\PiM) \Phi \right\|_\Lmu + \left\| \PiM \calL u_M \right\|_\Lmu \, \left| \lang \Phi, u_M \rang \right| ,
  \end{aligned}
  \end{equation}
  which allows to conclude.
\end{proof}

There are several ways to bound the right-hand side of~\eqref{eq:consistency_error_estimate}. It is difficult to state general results, and the strategy to be used depends on the model under consideration. One straightforward manner is to write 
\[
\| \PiM \calL (1-\PiM) \Phi \|_\Lmu \leq \left\| \PiM \calL (1-\PiM) \right\|_{\calB(\Htwomu, \Lmu)} \left\| (1-\PiM) \Phi \right\|_\Htwomu,
\]
and make use of the following (possible quite crude) bound which is independent of~$M$: 
\[
\left\| \PiM \calL (1-\PiM) \right\|_{\calB(\Htwomu, \Lmu)} \leq \left\| \calL \right\|_{\calB(\Htwomu, \Lmu)}.
\]
It remains then to show that the approximation error measured in the $\Htwomu$ norm goes to zero. Possibly sharper estimates can be obtained by writing that
\begin{equation}
  \label{eq:consistency L2}
  \| \PiM \calL (1-\PiM) \Phi \|_\Lmu \leq \left\| \PiM \calL (1-\PiM) \right\|_{\calB(\Lmu)} \left\| (1-\PiM) \Phi \right\|_\Lmu,
\end{equation}
and showing that $\left\| \PiM \calL (1-\PiM) \right\|_{\calB(\Lmu)}$ does not go too fast to infinity as $M$ goes to infinity. We can then conclude in the case when the approximation error vanishes sufficiently fast in $\Lmu$. This is the path we follow in Section~\ref{app:proof discrete hypocoercivity}.

\begin{remark}
\label{rmk:consistency error}
We expect the operator $\Picap \calL \Picap$ to be larger in a certain sense than $\Picap \calL (1-\Picap)$ in $\Lmu$, so that~\eqref{eq:relation approx consist errors} suggests that the consistency error is smaller than the approximation error $\|(1-\PiM)\Phi\|$. This is indeed what we observe in the numerical experiments we present in Figure~\ref{fig:H2 errors}. This shows that the way we bound the consistency error is probably not as sharp as it could be. 
\end{remark}

%
%
%

\subsection{Matrix conditioning and linear systems}
\label{ss:matrices}

We introduce in this section the linear system associated with the practical implementation of either the Galerkin formulation~\eqref{eq:weak formulation} in the conformal case $\VM \subset \tLmu$, or of~\eqref{eq:saddle point problem} in the non-conformal case $\VM \subset \Lmu$ but $\VM \not \subset \tLmu$. In any case, we denote by $(e_j)_{1 \leq j \leq M}$ an orthogonal basis of the Galerkin space~$\VM$, assumed to be of dimension~$M$. 

\paragraph{Conformal case.}
The weak formulation~\eqref{eq:weak formulation} can be equivalently reformulated as the linear system
\begin{equation}
  \label{eq:linear system poisson}
  \bfL_M \bfX_M = \bfY_M,
\end{equation}
where
\[
\forall 1 \leq i,j \leq M, \qquad \left( \bfL_M \right)_{i,j} = \lang e_i, -\calL e_j \rang, \quad 
\left( \bfX_M \right)_i = \lang \Phi_M, e_i \rang, \quad
\left( \bfY_M \right)_i = \lang R, e_i \rang.
\]
When the assumptions of Theorem~\ref{th:eq discrete hypocoercivity} hold, \eqref{eq:linear system poisson} admits a unique solution, so that $\bfL_M$ is invertible. Moreover $\| \bfL_M^{-1} \| \leq C/\lambdagammaM$ is bounded uniformly in~$M$ for $M \geq M_0$. The linear system is therefore well-conditioned, and can be solved efficiently using any solver adapted to non-symmetric problems.

\paragraph{Non-conformal case.}
We suppose that the assumptions of Corollary~\ref{coro:non-conformal hypocoercivity} hold. Let us introduce the vector $\bfU_M \in \bbR^M$ corresponding to $u_M \in \VM$:
\[
\forall 1 \leq i \leq M, \qquad \left( \bfU_M \right)_i = \lang u_M, e_i \rang = \frac{\lang \bfone, e_i \rang}{\| \PiM \bfone\|}.
\]
Then the saddle-point problem~\eqref{eq:lagrangian formulation} is equivalent to
\[
\left\{ \begin{aligned}
  \bfL_M \bfX_M + \lambda \bfU_M &= \bfY_M, \\
  \bfU_M^\top \bfX_M &= 0,
\end{aligned} \right.
\]
with the same definition for $\bfL_M$ and $\bfY_M$ as in the conformal case. With
\begin{equation}
  \label{eq:widehat_L_M}
  \widehat \bfL_M = \begin{pmatrix}
    & & &\vline & \\
    & \bfL_M & & \vline & \bfU_M \\
    & & &\vline &\\
    \hline
    & \dps \phantom{\int}\bfU_M^\top & & \vline & 0 
  \end{pmatrix},
  \qquad
  \widehat \bfX_M = \begin{pmatrix}
    \\
    \bfX_M \\
    \\
    \hline
    \lambda
  \end{pmatrix},
  \qquad
  \widehat \bfY_M = \begin{pmatrix}
    \\
    \bfY_M \\
    \\
    \hline
    0
  \end{pmatrix},
\end{equation}
the saddle-point problem can finally be rewritten as 
\[
\widehat \bfL_M \widehat \bfX_M = \widehat \bfY_M.
\]
Proposition~\ref{prop:saddle point} and~\eqref{eq:bound inverse tcalL} imply that $\widehat \bfL_M$ is invertible, with $\left\| \widehat \bfL_M^{-1} \right\|$ uniformly bounded in $M$ for $M \geq M_0$. This proves that the matrix $\widehat \bfL_M$ does not have vanishing eigenvalues, in contrast to~$\bfL_M$ (since $\bfL_M \bfU_M \xrightarrow[M \to \infty]{} 0$). Therefore the linear system $\widehat \bfL_M \widehat \bfX_M = \widehat \bfY_M$ can be solved as efficiently as in the conformal case. In the following we choose to use a sparse LU factorization.

\begin{remark}
Let us conclude this section with some criteria discriminating a good Galerkin space, and more generally a good function basis. Anticipating on the analysis of Section~\ref{ss:system and galerkin}, a standard choice is to use tensorized bases. The difficult part is to find a basis to describe the position dependence of the function of consideration. This requires considering the following points:
\begin{itemize}
\item approximation errors and consistency errors should be small. It should be checked in particular that condition~\eqref{eq:conditions hypocoercivity} holds and that the norm of the operator $\PiM \calL (1-\PiM)$ does not grow too fast.
\item the implementation is easier if the space is conformal, since it avoids the computation of $\bfU_M$ using integral quadratures.
\item when the basis is non-orthogonal, the Gram matrix should be inverted. The latter can be ill conditioned, leading to numerically instability, specifically for unbounded position spaces.
\end{itemize}
\end{remark}

\section{Application to a simple one-dimensional system}
\label{s:eq appli}

We present in this section an application of the theory developed in Section~\ref{s:discrete convergence} to a specific example, described in Section~\ref{ss:system and galerkin} together with the Galerkin basis used to discretize the generator. This allows us to prove explicit convergence rates for the approximation error (Section~\ref{ss:approximation error}) and the consistency error (Section~\ref{sec:consistency_specific_model}). For the latter error, we have to further specify the potential in order to check the assumptions ensuring the hypocoercivity of the discretized generator. The final, global error estimate is summarized in~\eqref{eq:total error}. The technical proofs of some claims and bounds are postponed to Appendix~\ref{app:proof discrete hypocoercivity}. We finally present in Section~\ref{s:application langevin} some numerical results illustrating the predicted error bounds.

\subsection{Description of the system and the Galerkin space}
\label{ss:system and galerkin}

We consider a single particle in a one-dimensional periodic potential: $D=1$, $m=1$ and $\calD = 2\pi \bbT = \bbR / 2 \pi \bbZ$. The Galerkin space is constructed using the spectral tensor basis
\[
e_{k,\ell}(q,p) = G_k(q) H_\ell(p),
\]
where $0 \leq k < 2K-1$ and $0 \leq \ell < L$. Compared to the notation of Section~\ref{s:discrete convergence}, the basis size $M = (2K-1)L$ depends on two parameters~$K,L$, which both have to go to infinity for the convergence results to hold. In this section we prefer the index $KL$ instead of $M$, denoting thus $\VKL$, $\PiKL$, $\PhiKL$,... In the remainder of this section we describe our choices for $G_k$ and $H_\ell$.

Note that the size of the matrix, namely the number of tensorized basis elements, increases exponentially with the dimension of the system. In larger dimension one could consider resorting to tensor formats~\cite{Hackbusch12}, as is done for the high-dimensional Schrödinger equation in~\cite{Yserentant10}, carefully making use of the symmetries and of the structure of the equation.

\paragraph{Weighted Fourier basis $(G_k)$.}
Fourier modes provide a natural basis to approximate periodic functions, such as functions of the positions here. Since the measure appearing in the scalar product is $\nu$, we consider in fact the following $L^2(\nu)$-orthonormal modes:
\begin{equation}
  \label{eq:functions g_k}
  \begin{aligned}
    G_0(q) &= \sqrt{\frac{Z_{\beta, \nu}}{2\pi}} \, \rme^{\beta V(q)/2}, \\
    G_{2k}(q) &= \sqrt{\frac{Z_{\beta, \nu}}{\pi}} \cos(kq) \, \rme^{\beta V(q)/2}, \qquad k \geq 1,\\
    G_{2k-1}(q) &= \sqrt{\frac{Z_{\beta, \nu}}{\pi}} \sin(kq) \, \rme^{\beta V(q)/2}, \qquad k \geq 1.
  \end{aligned}
\end{equation}
Note that the functions $G_k$ for $k \geq 1$ do not have mean~0 with respect to~$\nu$ (except for very specific potentials such as $V=0$). The spanned discretization space is thus non-conformal: $\VKL \not\subset \tLmu$.

\paragraph{Hermite functions basis $(H_\ell)$.}
Since the marginal measure $\kappa$ in the momentum variables is Gaussian with variance $\betainv$, we consider the following orthonormal Hermite modes for $\ell \in \bbN$:
\[
H_\ell(p) =  \frac 1 {\sqrt{\ell!}} \widetilde H_\ell \left(\sqrt{ \beta} p \right), 
\qquad 
\widetilde H_\ell (y) =  (-1)^\ell \rme^\frac{y^2}{2} \frac{d^\ell}{dy^\ell} \left( \rme^{-\frac{y^2}{2}} \right).
\]
They are well suited to our problem since they are the eigenfunctions of the symmetric part $\LFD = - \beta^{-1} \partial_p^* \partial_p$ of the generator. Indeed, 
\begin{equation}
  \label{eq:derivatives H}
  \forall \ell \in \bbN, \qquad \partial_p H_\ell = \sqrt{\beta \ell} H_{\ell-1} \qquad \mbox{and} \quad \partial_p^* H_\ell = \sqrt{\beta (\ell+1)} H_{\ell+1},
\end{equation}
so that
\begin{equation}
  \label{eq:LFD hermite}
  \forall \ell \in \bbN, \qquad \LFD H_\ell = -\ell H_\ell.
\end{equation}

\begin{remark}
The basis we consider is similar to the one used in~\cite{Risken96} and~\cite{PavVog08}, where the modes in position are the standard Fourier modes. The latter modes are orthogonal for the uniform measure on the compact position space $\calD$ rather than on~$\Lnu$. Therefore, the scalar product used in Subsection~\ref{ss:matrices} should be replaced with the scalar product associated with the measure $\widetilde{\mu}(\dd q \, \dd p) = |\mathcal{D}|^{-1} \kappa(\dd p) \, \dd q$. The results of Section~\ref{s:discrete convergence}  could be adapted to this scalar product since the measures $\mu$ and $\widetilde{\mu}$ are equivalent. Note that the discretization based on the standard Fourier modes is a conformal one since one of the tensorized modes is proportional to~$\bfone$, which simplifies the implementation. It is however not generalizable to unbounded position spaces because the uniform measure is not normalizable. An interesting question, not considered in this work, is to quantify the relative performances of the approaches based on orthonormal bases either on $\Lmu$ or $\rmL^2(\widetilde{\mu})$.
\end{remark}

\paragraph{Rigidity matrix.}
In order to give the expression of the rigidity matrix, we introduce, for a Fourier basis of $2K-1$ weighted Fourier modes, the matrix $\bfQ$ with entries
\begin{equation}
  \label{eq:Q matrix}
  \bfQ_{k,k'} = \lang G_k, \partial_q G_{k'} \rang_\Lnu,
\end{equation}
and, for $L$ Hermite modes, the matrix $\bfP$ with entries
\begin{equation}
  \label{eq:matrix P}
  \bfP_{\ell, \ell'} = \lang H_\ell, \partial_p H_{\ell'} \rang_\Lkappa = \lang H_\ell, \sqrt{\beta \ell'} H_{\ell'-1} \rang_\Lkappa = \sqrt{\beta \ell'} \delta_{\ell, \ell'-1}.
\end{equation}
Note that $\bfP$ is sparse in view of~\eqref{eq:derivatives H}. The matrix $\bfQ$ is, on the other hand, dense in general, except when $V$ is a trigonometric polynomial. In the following, we choose $V(q) = 1 - \cos(q)$ in order for $\bfQ$ to be tridiagonal. For a general, smooth potential~$V$, $\bfQ$ would be dense but with coefficients which decay fast away from the diagonal. 

The rigidity matrix which appears on the left-hand side of~\eqref{eq:linear system poisson} has entries (for $0 \leq k \leq 2K-2$ and $0 \leq \ell \leq L-1$)
\[
\begin{aligned}
  \bfL_{k \ell, k' \ell'} &= \lang e_{k \ell}, -\calL e_{k' \ell'} \rang \\
  &= -\betainv \left[ \lang G_k H_\ell, \paq \pap^* G_{k'} H_{\ell'} \rang
    -\lang G_k H_\ell, \paq^* \pap G_{k'} H_{\ell'} \rang
    -\gamma \lang G_k H_\ell, \pap^* \pap G_{k'} H_{\ell'} \rang \right] \\
  &=-\betainv \bfQ_{k, k'} \bfP_{\ell',\ell} + \betainv \bfQ_{k',k} \bfP_{\ell,\ell'} + \gamma \bfI_{k,k'} \bfN_{\ell,\ell'},
\end{aligned}
\]
where $\bfI_{k,k'} = \delta_{k, k'}$ and $\bfN_{\ell,\ell'} = \ell \, \delta_{\ell, \ell'}$. In practice we transform these tensors into matrices by a hashing function $\zeta : (k,\ell) \to \zeta(k,\ell) \in \bbN$. The matrix $\bfL$ is then of size $(2K-1)L$. 

\subsection{Approximation error for the tensor basis}
\label{ss:approximation error}

\newcommand{\indDv}{s}

We define the projectors $\PiK$ and $\PiL$ by
\[
\PiK \varphi = \sum_{k = 0}^{2K-2} \lang \varphi, G_k \rang G_k, \qquad \PiL \varphi = \sum_{\ell = 0}^{L-1} \lang \varphi, H_\ell \rang H_\ell.
\]
Their complements are $\PiKo = 1-\PiK$ and $\PiLo = 1-\PiL$. With this notation, the projector onto the Galerkin space is $\PiKL = \PiK \PiL$. The study of the approximation error $(1-\PiKL)\Phi$ is performed by first estimating the error arising from the projection~$\PiK$ (see Lemma~\ref{lem:approx_PiK}), and then the error arising from~$\PiL$ (see Lemma~\ref{lem:approx_PiL}). The conclusion follows by remarking that 
\begin{equation}
  \label{eq:ineq_PiKL}
  0 \leq 1 - \PiKL = 1-\PiK + \PiK (1-\PiL) \leq \PiKo + \PiLo,
\end{equation}
see Proposition~\ref{prop:approximation error}.

\begin{lemma}
  \label{lem:approx_PiK}
  Assume that $V$ is smooth. Then, for any $\indDv \in \bbN$, there exists $M_s \in \bbR_+$ such that
  \[
  \forall \varphi \in \rmH^\indDv(\nu), \quad \forall K \geq 1, 
  \qquad 
  \left\| \varphi - \PiK \varphi \right\|_\Lnu \leq \frac {M_s}{K^\indDv} \| \varphi \|_{\rmH^\indDv(\nu)}.
  \]
\end{lemma}

\begin{proof}
For $\varphi \in \rmH^\indDv(\nu)$, we introduce $\widetilde{\varphi} = Z_{\beta,\nu}^{-1/2}\rme^{-\beta V/2} \varphi \in \rmL^2(\dd q)$, as well as the flat Fourier basis $\widetilde{G}_k = Z_{\beta,\nu}^{-1/2}\rme^{-\beta V/2} G_k$ which is orthonormal on $\rmL^2([0,2\pi])$. Since $\calD = 2 \pi \bbT$ is compact, $\rmH^\indDv(\nu) = \rmH^\indDv(\dd q)$ for any $\indDv \in \bbN$ and there exists $M_s \in \bbR_+$ such that 
\begin{equation}
  \label{eq:ineg_Ms}
  \| \paq^\indDv \varphi \|_{\rmL^2(\dd q)} \leq M_\indDv \| \varphi \|_{\rmH^\indDv(\nu)}. 
\end{equation}
By the Bessel-Parseval inequality, 
\[
\begin{aligned}
  \| \varphi - \PiK \varphi \|_\Lnu^2 &= \sum_{k \geq 2K-1} \lang \varphi, G_k \rang^2
  = \sum_{k \geq 2K-1}  \left( \int_0^{2\pi} \widetilde{\varphi} \, \widetilde{G}_k \, \dd q \right)^2 \\
  &= \frac1\pi \sum_{k \geq K}  \left(  \int_0^{2\pi} \widetilde{\varphi}(q) \, \cos(kq) \, \dd q \right)^2 + \left( \int_0^{2\pi} \widetilde{\varphi}(q) \, \sin(kq) \, \dd q \right)^2\\
  &\leq \frac1\pi \sum_{k \geq K}  \left( \int_0^{2\pi} \widetilde{\varphi}(q) \, \frac {k^\indDv}{K^\indDv} \cos(kq) \, \dd q \right)^2 + \left( \int_0^{2\pi} \widetilde{\varphi}(q) \, \frac {k^\indDv}{K^\indDv} \sin(kq) \, \dd q \right)^2\\
  &= \frac{1}{\pi K^{2\indDv}} \sum_{k \geq K}  \left( \int_0^{2\pi} \widetilde{\varphi}(q) \, \paq^\indDv \cos(kq) \, \dd q \right)^2 + \left( \int_0^{2\pi} \widetilde{\varphi}(q) \, \paq^\indDv \sin(kq) \, \dd q \right)^2\\
  &\leq \frac{1}{K^{2\indDv}} \| \paq^\indDv \widetilde{\varphi} \|_{\rmL^2(\dd q)}^2,
\end{aligned}
\]
which allows to conclude with~\eqref{eq:ineg_Ms}.
\end{proof}

\begin{lemma}
  \label{lem:approx_PiL}
  For any $\indDv \in \bbN$ and $\varphi \in \rmH^\indDv(\kappa)$, it holds
  \[
  \forall L \geq \indDv, \qquad \left\| \varphi - \PiL \varphi \right\|_\Lkappa \leq \left[ \beta (L-\indDv+1) \right]^{-\indDv/2} \| \pap^\indDv \varphi \|_\Lkappa.
  \]
\end{lemma}

\begin{proof}
Fix $L \geq \indDv$. In view of~\eqref{eq:derivatives H}, it holds 
\[
\left(\partial_p^*\right)^s H_{\ell-s} = \beta^{s/2} \sqrt{(\ell-s+1)\dots \ell} \, H_\ell,
\]
with $\sqrt{(\ell-s+1)\dots \ell} \geq (L-s+1)^{s/2}$ when $\ell \geq L$. Therefore,
\[
\begin{aligned}
  \| \varphi - \PiL \varphi \|_\Lkappa^2 = \sum_{\ell \geq L} \lang \varphi, H_\ell \rang^2 & \leq \sum_{\ell \geq L} \lang \varphi, \frac {\sqrt{(\ell-s+1)\dots \ell}}{(L-\indDv+1)^{\nicefrac \indDv 2}} H_\ell \rang^2 \\
  &= \left[ \beta (L-\indDv+1) \right]^{-\indDv} \sum_{\ell \geq L} \lang \varphi, (\pap^*)^\indDv H_{\ell-\indDv} \rang^2 \\
  &\leq \left[ \beta (L-\indDv+1) \right]^{-\indDv} \| \pap^\indDv \varphi \|_\Lkappa^2,
\end{aligned}
\]
from which the conclusion follows.
\end{proof}

The following approximation result is then directly deduced from the previous lemmas and~\eqref{eq:ineq_PiKL}.

\begin{prop}
\label{prop:approximation error}
Assume that $V$ is smooth. Then, for any $s \in \bbN$, there exists $A_s \in \bbR_+$ such that
\[
\forall \varphi \in \Hsmu, \qquad \forall K \geq 1, \ L \geq s, 
\qquad
\| \varphi - \PiKL \varphi \|_\Lmu \leq A_s \left(\frac{1}{K^s} + \frac{1}{L^{s/2}}\right) \|\varphi\|_{\Hsmu}.
\]
\end{prop}

The approximation error $\|(1-\PiKL)\Phi\|$ thus depends on the regularity of the solution~$\Phi$ of the Poisson problem. Now, the operator $\calL^{-1}$ is a bounded operator on $\rmH^s(\mu) \cap \tLmu$ for any $s \geq 0$ by the results of~\cite[Section~3.2]{Talay02} and~\cite{Kopec15} (see also~\cite{Eckmann03,Herau04}). Therefore, when $R \in \rmH^s(\mu) \cap \tLmu$, the solution $\Phi$ belongs to $\rmH^s(\mu) \cap \tLmu$, and there is $\widetilde{A}_s \in \bbR_+$ such that
\begin{equation}
  \label{eq:approximation error solution}
  \| \Phi - \PiKL \Phi\|_\Lmu \leq \| \Phi - \PiK \Phi\|_\Lmu + \| \Phi - \PiL \Phi\|_\Lmu \leq \widetilde{A}_s \left(\frac{1}{K^s} + \frac{1}{L^{s/2}}\right) \| R \|_{\Hsmu}.
\end{equation}

\begin{remark}
  In fact, it can be expected that the operator $\calL^{-1}$ further regularizes in the momentum variable; more precisely that $\partial_p \Phi \in \rmH^s(\mu)$ when $R \in \rmH^s(\mu)$. This is consistent with what we observe in the numerical simulations reported in Section~\ref{s:application langevin}. Note also that the estimates provided by~\cite{Talay02,Eckmann03,Herau04,Kopec15} are obtained for a fixed friction~$\gamma>0$. Some additional work is needed to carefully quantify their dependence upon~$\gamma$, although we expect that the bounds on $\calL^{-1}$ considered as an operator on $\rmH^s(\mu) \cap \tLmu$ should still scale as $\max(\gamma,\gamma^{-1})$. 
\end{remark}

Let us conclude this section by an approximation result involving $\Pi_{KL,0}$ rather than $\PiKL$ (see the decomposition~\eqref{eq:decomposition of the error non conformal}, to be compared with~\eqref{eq:decomposition of the error}). 

\begin{corollary}
  Assume that $V$ is smooth. Then, for any $s \in \bbN$, there exists $A_s \in \bbR_+$ such that
  \[
  \forall \varphi \in \Hsmu \cap \tLmu, \qquad \forall K \geq 1, \ L \geq s, 
  \qquad
  \| \varphi - \Pi_{KL,0} \varphi \|_\Lmu \leq A_s \left(\frac{1}{K^s} + \frac{1}{L^{s/2}}\right) \|\varphi\|_{\Hsmu}.
  \]
\end{corollary}

\begin{proof}
Note first that $\lang H_\ell, \bfone \rang = \delta_{\ell,0}$, so that $\PiKL \bfone = \PiK \bfone$ and $u_K = \PiK \bfone/\| \PiK \bfone \|$ depends only on the position variables for $L \geq 1$. Next, in view of the computations performed in the proof of Corollary~\ref{coro:non-conformal hypocoercivity},
\[
\Pi_{KL,0} \varphi = \PiKL \varphi - \lang \frac{\PiK \bfone}{\|\PiK \bfone\|},\varphi\rang u_K,
\]
where $\|u_K\|=1$. Since $\varphi \in \tLmu$, it holds in fact
\[
\lang \frac{\PiK \bfone}{\|\PiK \bfone\|},\varphi\rang = \lang \frac{(1-\PiK) \bfone}{\|\PiK \bfone\|},\varphi\rang,
\]
which converges to~0 faster than any polynomial in $K$ in view of Proposition~\ref{prop:approximation error}.
\end{proof}

\subsection{Consistency error}
\label{sec:consistency_specific_model}

In order to simplify the computations (in particular to have some simple structure on the derivatives of the Fourier modes) we consider the following potential:
\[
V(q) = 1 - \cos(q).
\]
In this case, using the trigonometric identities
\[
\begin{aligned}
  2 \cos(kq) \sin(q) &= \sin((k+1)q) - \sin((k-1)q) \\
  2 \sin(kq) \sin(q) &= -\cos((k+1)q) + \cos((k-1)q),
\end{aligned}
\]
a straightforward computation shows that the derivatives of the basis functions satisfy 
\begin{equation}
  \displaystyle
  \label{eq:derivatives G}
  \begin{array}{rlrl}
    \dps \partial_q G_0 &\hspace{-6pt}= \dps \frac \beta {2 \sqrt 2} G_1, &\qquad
    \dps \partial_q G_1 &\hspace{-6pt}= \dps \frac \beta {2 \sqrt 2} G_0 + G_2 - \frac \beta 4 G_4, \\[10pt]
    \dps \partial_q G_{2k} &\hspace{-6pt}= \dps -\frac \beta 4 G_{2k-3} - k G_{2k-1} + \frac \beta 4 G_{2k+1}, &\qquad
    \dps \partial_q G_{2k-1} &\hspace{-6pt}= \dps\frac \beta 4 G_{2k-2} + k G_{2k} - \frac \beta 4 G_{2k+2},
  \end{array}
\end{equation}
where by convention $G_{-1} = 0$. The matrix $\bfQ$ defined in~\eqref{eq:Q matrix} is therefore a band matrix with width $4$. 

The well-posedness of the variational formulation associated with the Galerkin space is given by the following result. 

\begin{prop}
\label{prop:hypocoercivity sinus}
The matrix $\widehat \bfL_{KL}$ defined in~\eqref{eq:widehat_L_M} is invertible for $K, L$ sufficiently large. More precisely the resolvent bound satisfies
\begin{equation}
  \label{eq:estimate_tlambdagammaM_explicit}
  \tlambdagammaKL \geq \lambda_\gamma - \frac \varepsilon {1+\varepsilon} \left[ \frac {(1+ \sqrt 2) \beta} {2 K} + \frac {\beta^3} {16} \frac{ \| (1-\PiKm) \bfone \|^2}{1 - \| (1-\PiK) \bfone \|^2} \right].
\end{equation}
\end{prop}

In practice the term $\| (1-\PiKm) \bfone \|$ is very small (it decays faster than any polynomial in~$K$ by Lemma~\ref{lem:approx_PiK}), so that the difference between the two estimates $\lambda_\gamma - \tlambdagammaKL$ scales as $1/K$ and in particular it does not depend on $L$. The proof presented in Appendix~\ref{app:proof discrete hypocoercivity} consists in showing that the assumptions of Corollary~\ref{coro:non-conformal hypocoercivity} hold. Recall also that $\varepsilon, \lambdagamma \sim \min(\gamma, \gamma^{-1})$ by Proposition~\ref{prop:coercivity_scrD}, so that the error term on the right-hand side of~\eqref{eq:estimate_tlambdagammaM_explicit} is uniformly bounded with respect to~$\lambda_\gamma$. This suggests that the relative error on the spectral gap is uniformly bounded with respect to~$\gamma > 0$.

\medskip

According to Theorem~\ref{th:consistency error} and~\eqref{eq:norm Lpm} the following rate of convergence can be deduced for the error $\PhiM - \PiKLnot \Phi$ (which is related to the consistency error $\PiKLnot \calL \PiKLnot \Phi + \PiKLnot R$). 

\begin{prop}
\label{prop:consistency sinus}
The error $\|\PhiKL - \PiKLnot \Phi\|$ is bounded by the approximation error as
\[
\begin{aligned}
  \| \PhiKL - \PiKLnot \Phi \|_\Lmu \leq \frac C\tlambdagammaKL \left[ \sqrt{\frac{L}{\beta}} \left( K-1+\beta \right) \| (1-\PiKL) \Phi \|_\Lmu + \| \calL u_K \|  \| \Phi \|_\Lmu \right].
\end{aligned}
\]
where $\| \calL u_K \|$ decays faster than any polynomial (see \eqref{eq:L u_M} for an explicit computation). Therefore, for any $s \geq 1$, there exists $A_{\gamma, s} \in \bbR_+$ such that, for all $R \in \rmH^s(\mu)$ and $\Phi = -\calL^{-1} \Pinot R$, 
\[
\| \PhiKL - \PiKLnot \Phi \|_\Lmu \leq A_{\gamma, s} K \sqrt L \left(\frac{1}{K^s} + \frac{1}{L^{s/2}}\right)\| R\|_{\Hsmu}. 
\]
\end{prop}

The second statement follows from the bounds on the approximation error $\| \Phi - \PiKLnot \Phi \|_\Lmu$ provided by Proposition~\ref{prop:approximation error}, together with the fact that $\calL^{-1}$ is a bounded operator on $\rmH^s(\mu) \cap \tLmu$ (see the discussion at the end of Section~\ref{ss:approximation error}). The total error can thus be bounded as
\begin{equation}
\label{eq:total error}
\begin{aligned}
	\| \PhiKL - \Phi \| &\leq \| \PhiKL - \PiKLnot \Phi \| + \| \Phi - \PiKLnot \Phi \| \\
	&\leq A_{\gamma,s} K \sqrt L \left(\frac{1}{K^s} + \frac{1}{L^{s/2}}\right)\| R\|_{\Hsmu} + \widetilde{A}_s \left(\frac{1}{K^s} + \frac{1}{L^{s/2}}\right) \| R \|_{\Hsmu} \\
	&\leq \widehat A_{\gamma,s} K \sqrt L \left(\frac{1}{K^s} + \frac{1}{L^{s/2}}\right)\|R\|_{\Hsmu}.
\end{aligned}
\end{equation}

\subsection{Numerical results}
\label{s:application langevin}

In this section we call for simplicity consistency error the quantity $\| \PhiKL - \PiKL \Phi \|$. In order to validate the results of Section~\ref{s:discrete convergence} in the non-conformal case studied here, we compute the consistency error and the approximation error $\| \Phi - \PiKL \Phi\|$ as a function of the number $K,L$ of modes and of the friction coefficient $\gamma$. We start by considering an observable which is not very regular; and then turn our attention to the case when $R(q,p) = p$ (which belongs to $\Hsmu$ for any $s \in \bbN$). Solving the Poisson equation associated with this observable allows to predict the self-diffusion coefficient, which can be seen as the magnitude of the effective Brownian motion describing Langevin dynamics over diffusive timescales~\cite{PavStu08}. In all this section we set $\beta = 1$ and $m=1$.

As a sanity check we also verified in the case $V=0$ that the eigenvalues of the rigidity matrix $\bfL$ converge to their analytical expressions provided in~\cite{Risken96}.

\paragraph{Observable nearly in $\rmH^2(\mu)$.}
Fix $\gamma=1$ and consider the observable 
\[
R = \sum_{k \in \bbN, \ell \in \bbN} r_{k\ell} G_k H_\ell, \qquad r_{k\ell} = \max(1,k)^{-5/2} \max(1, \ell)^{-3/2}.
\]
Note that 
\[
\|R\|^2 = \sum_{k \in \bbN, \ell \in \bbN} |r_{k\ell}|^2 < +\infty.
\]
Using~\eqref{eq:derivatives G} and~\eqref{eq:derivatives H} it can be shown that $R$ is in $\rmH^1(\mu)$ but fails to be in $\rmH^2(\mu)$ (the exponents in $r_{k\ell}$ are critical). Note also that $R$ does not have mean~0 with respect to~$\mu$, so that the solution of the saddle point problem~\eqref{eq:saddle point problem} converges to the solution of the Poisson problem with $\Pinot R$ on the right-hand side. A very accurate approximation of the solution $\Phi$, which serves as a reference value, is computed by setting $K = 100$ and $L = 1000$. The errors are plotted in Figure~\ref{fig:H2 errors}.

\begin{figure}[ht!]
\begin{center}
\includegraphics[scale=.5]{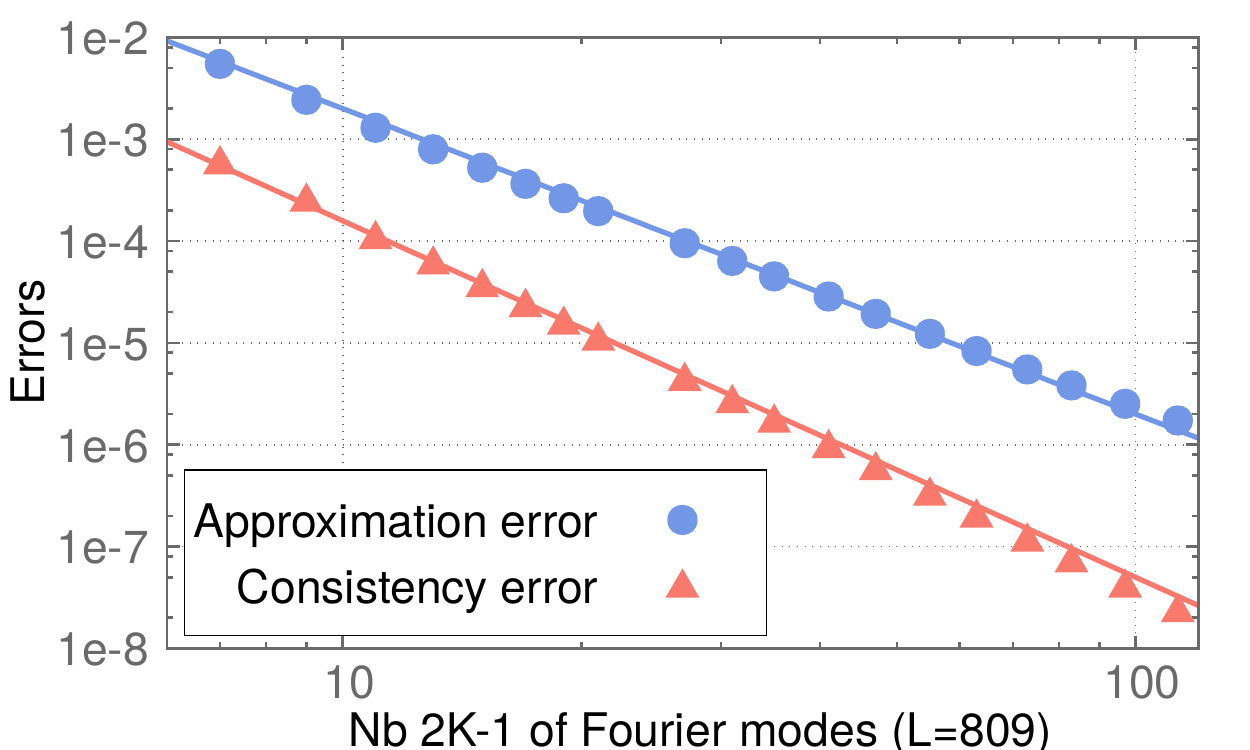}
\includegraphics[scale=.5]{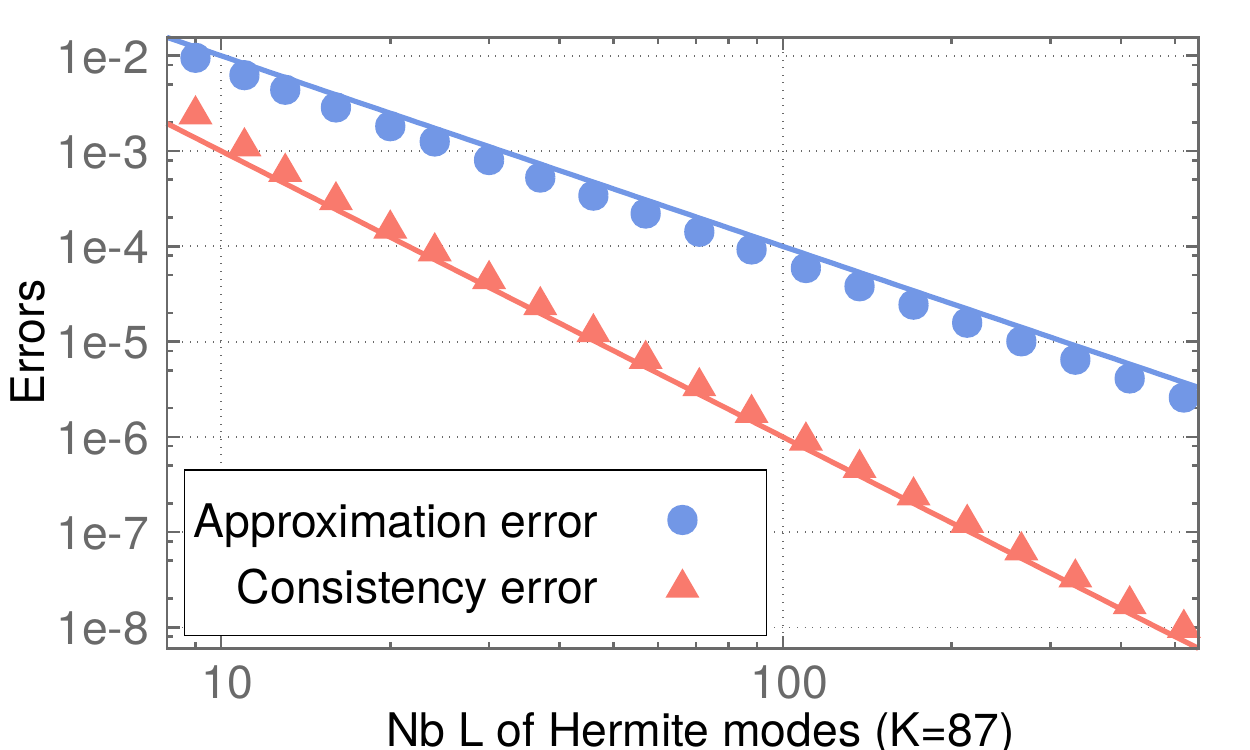}
\end{center}
\caption{Approximation and consistency errors as a function of the number of modes. Left: varying number of Fourier modes for a large number of Hermite modes; the approximation error scales as $K^{-3}$ while the consistency error scales as $K^{-7/2}$. Right: varying number of Hermite modes for a large number of Fourier modes; the approximation error scales as $L^{-2}$ while the consistency error scales as $L^{-3}$.}
\label{fig:H2 errors}
\end{figure}

The polynomial power of the numerically observed decay of the approximation error is directly linked to the regularity of the solution $\Phi$. Here the scalings $K^{-3}$ and $L^{-2}$ suggest that $\Phi, \pap \Phi \in \rmH^3(\mu)$, meaning that in this particular case $\calLinv$ regularizes one derivative of $R$ in position and two in momenta, which is the most that could be expected. Note that the approximation error is therefore much smaller than predicted in~\eqref{eq:approximation error solution}, where we only stated that $\Phi$ is at least as regular as $R$. Moreover, we observe that the consistency error decays faster than the approximation error, as anticipated in Remark~\ref{rmk:consistency error}.

\paragraph{Velocity observable.}

The self-diffusion of a particle subjected to Langevin dynamics in dimension~1 is (see for instance~\cite[Section~5]{Lelievre16} for further background)
\begin{equation}
  \label{eq:green kubo}
  D = \int_0^\infty \Esp \left( p_t p_0 \right) \dt = \lang -\calLinv p, p \rang,
\end{equation}
where the expectation is taken over all initial conditions $(q_0,p_0)\sim \mu$ and for all realizations of the Brownian motion in~\eqref{eq:langevin dynamics}. This transport coefficient can be computed by approximating $\Phi = \calLinv p$ with the Galerkin method described in this article. The accurate reference is here computed by setting $K = 50$ and $L = 100$. We plot on Figure~\ref{fig:velocity errors} the approximation error and the consistency error obtained for the observable $R(q,p) = p$. They decay faster than any polynomial since $p \in \Hsmu$ for any $s \in \bbN$. They are in fact observed to decay exponentially fast with the number of modes. The error on the self-diffusion coefficient therefore also decays faster than any polynomial, in fact exponentially. 

\begin{figure}[ht!]
\begin{center}
\includegraphics[scale=.5]{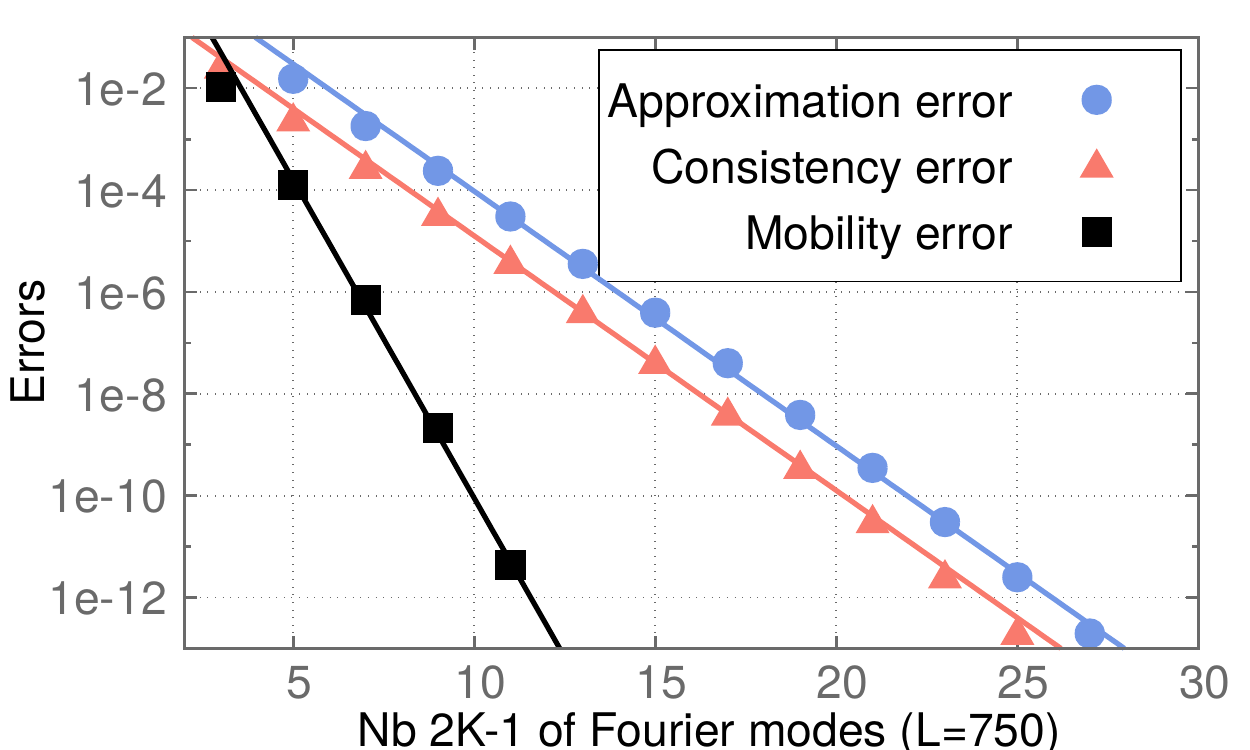}
\includegraphics[scale=.5]{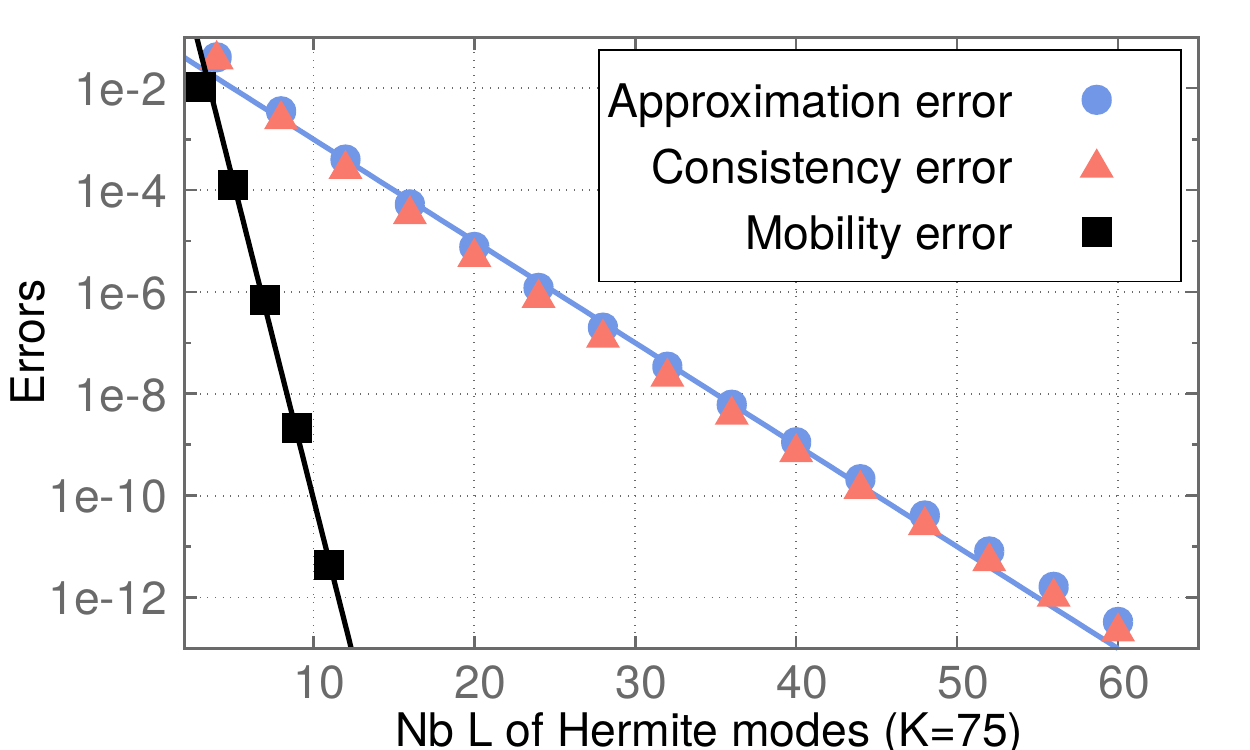}
\end{center}
\caption{Approximation error, consistency error and error on the mobility as a function of the number of Fourier modes (Left) or Hermite modes (Right) for $\gamma=1$. Logarithmic units are used on the ordinate axis. When the number of Hermite modes is large, the error on the mobility scales as $10^{-2.5 K}$, while the approximation and consistency errors both scale as $10^{-K}$. When the number of Fourier modes is large, the error on the mobility scales as $10^{-1.25 L}$, while the approximation and consistency errors both scale as $10^{-0.2 L}$.}
\label{fig:velocity errors}
\end{figure}

As an illustration of our approach, we plot the value of the self-diffusion as a function of~$\gamma$ in Figure~\ref{fig:mobility fct gamma}, as already done in~\cite{PavStu08} using Monte-Carlo techniques and in~\cite{PavVog08} using a very similar spectral method. We indeed retrieve the scaling $D \sim \gamma^{-1}$ proved in~\cite{PavStu08}. This computation can be done in a matter of seconds as it involves a single inversion of a sparse matrix of size $KL = 5000$ for each value of the friction $\gamma$. It is thus much faster that a standard Monte-Carlo simulation. This approach however becomes intractable when the dimension increases.

\begin{figure}[ht!]
\begin{center}
\includegraphics[scale=.5]{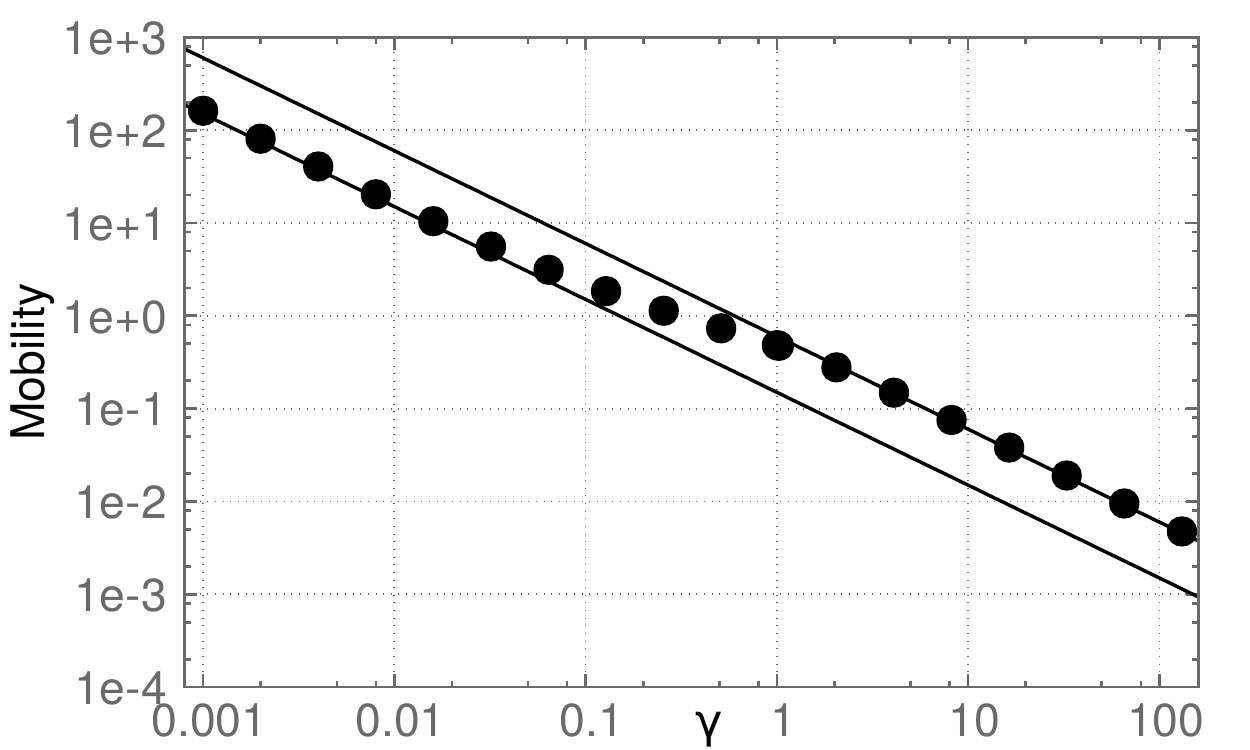}
\end{center}
\caption{Self-diffusion as a function of the friction~$\gamma$. It scales as $\gamma^{-1}$ both for small $\gamma$ (with prefactor $0.15$) and large $\gamma$ (with prefactor $0.6$).}
\label{fig:mobility fct gamma}
\end{figure}

\paragraph{Estimates on the spectral gap.}
In order to illustrate the statements of Proposition~\ref{prop:hypocoercivity sinus}, we compute the relative error between the spectral gap of $\calL$ (approximated using a very large discretization basis) and the spectral gap of the matrix $\widehat \bfL$; see Figure~\ref{fig:spectral gap wrt gamma}. The spectral gap is close to the value $\min(\gamma, \gamma^{-1})$ obtained when $V=0$ (see~\cite{Kozlov89}), with deviations essentially around $\gamma=1$. Note on Figure~\ref{fig:spectral gap wrt KL} that the relative error on the spectral gap decays exponentially with $K$ and $L$. Let us also emphasize that, as suggested by~\eqref{eq:estimate_lambda_gamma_M_unif}, the relative error on the spectral gap is bounded uniformly with respect to~$\gamma$ for any $K,L$. We also observe that in the overdamped limit $\gamma \to \infty$ the relative error depends only on the discretization accuracy in the position variable. This is due to the fact that the resolvent $\calL^{-1}$ converges in this regime to an operator acting only on the position variables~\cite{Leimkuhler16}.

\begin{figure}[ht!]
\begin{center}
\includegraphics[scale=.5]{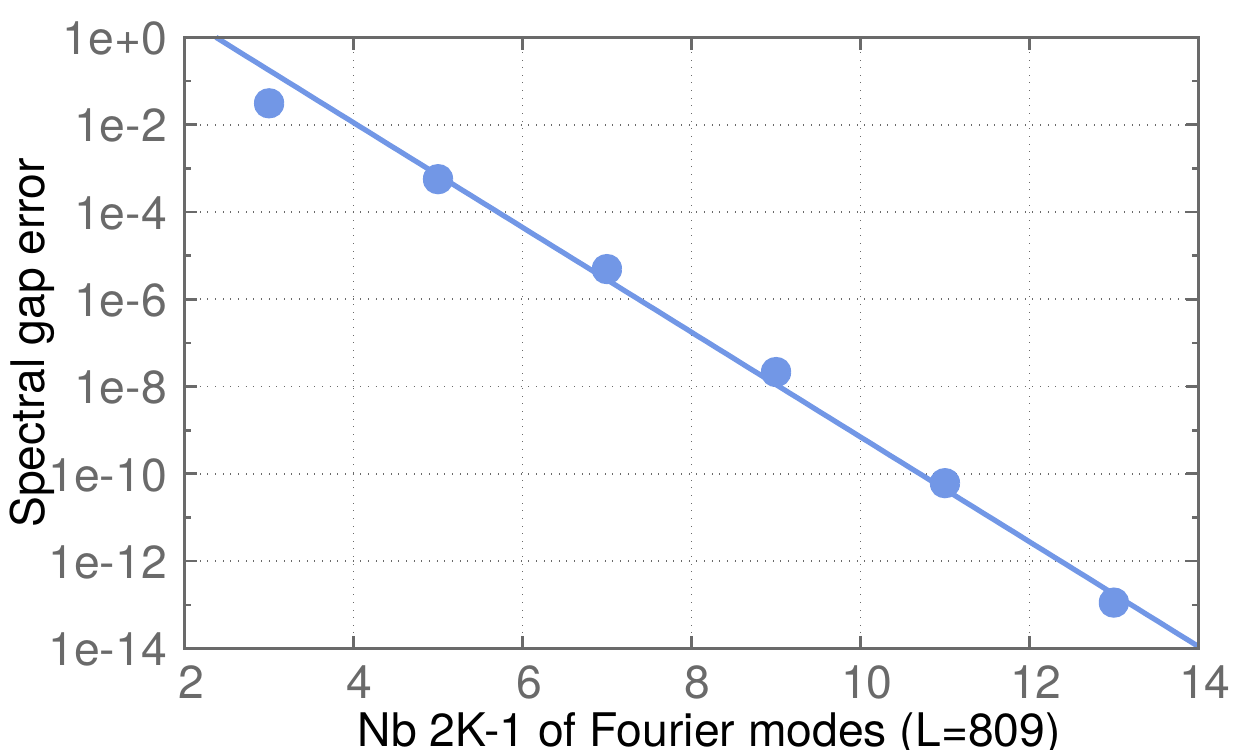}
\includegraphics[scale=.5]{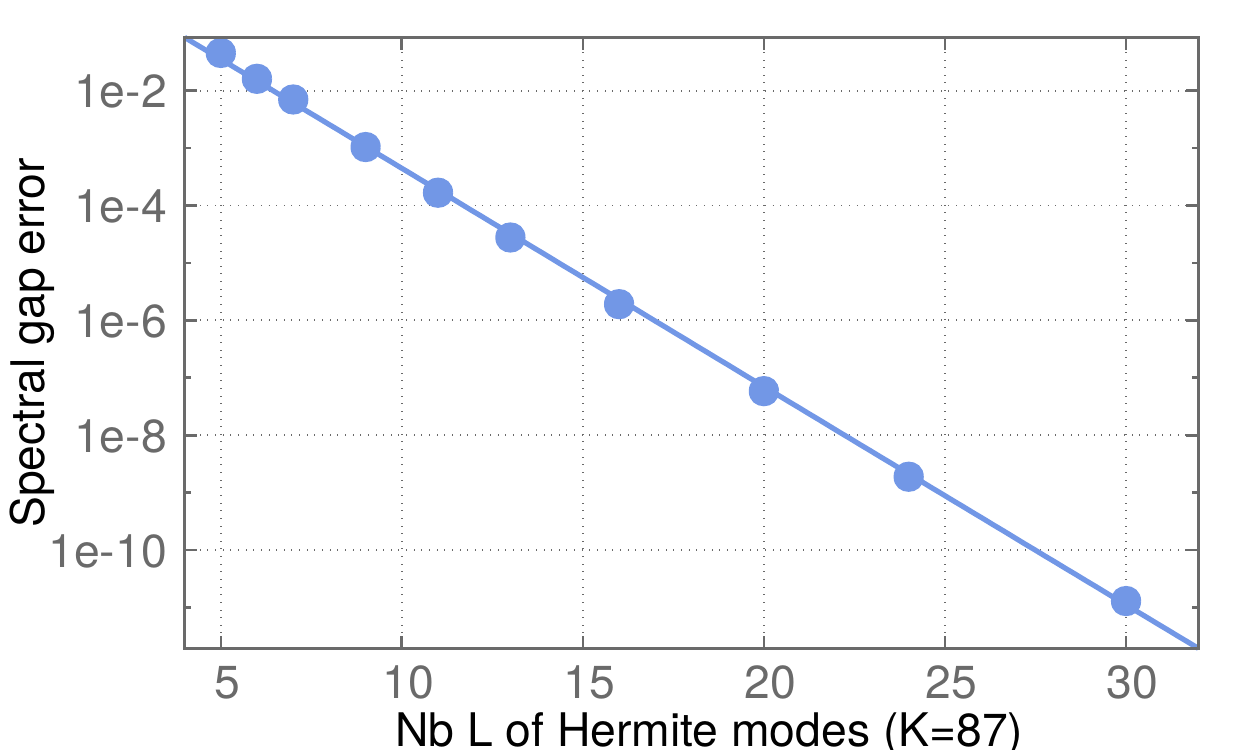}\\
\includegraphics[scale=.5]{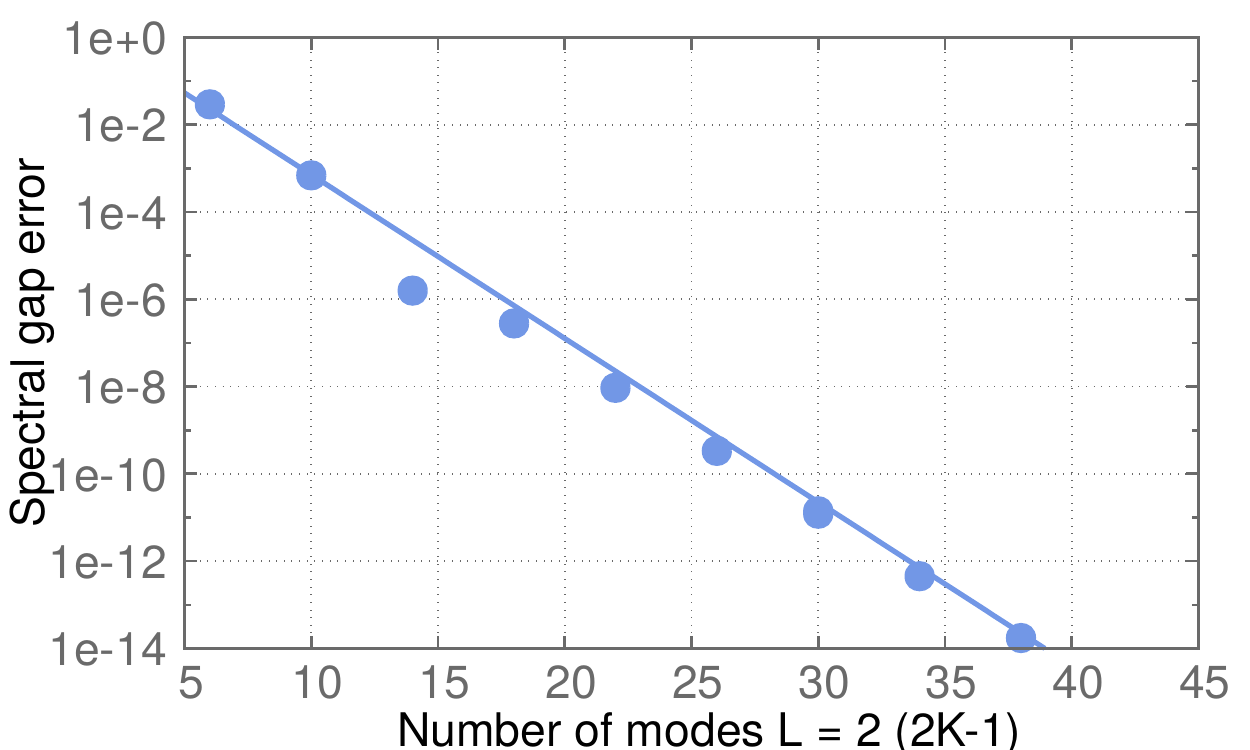}
\end{center}
\caption{Error on the spectral gap as a function of the size of the basis in three cases for $\gamma = 1$. For a large number of Hermite modes the error scales approximatively as $10^{-1.2 (2K-1)}$ (top left); for a large number of Fourier modes it scales approximatively as $10^{-0.32 L}$ (top right); and for $L = 2 (2K-1)$ it scales approximatively as $10^{-3.8 L}$ (bottom).}
\label{fig:spectral gap wrt KL}
\end{figure}

\begin{figure}[ht!]
\begin{center}
\includegraphics[scale=.5]{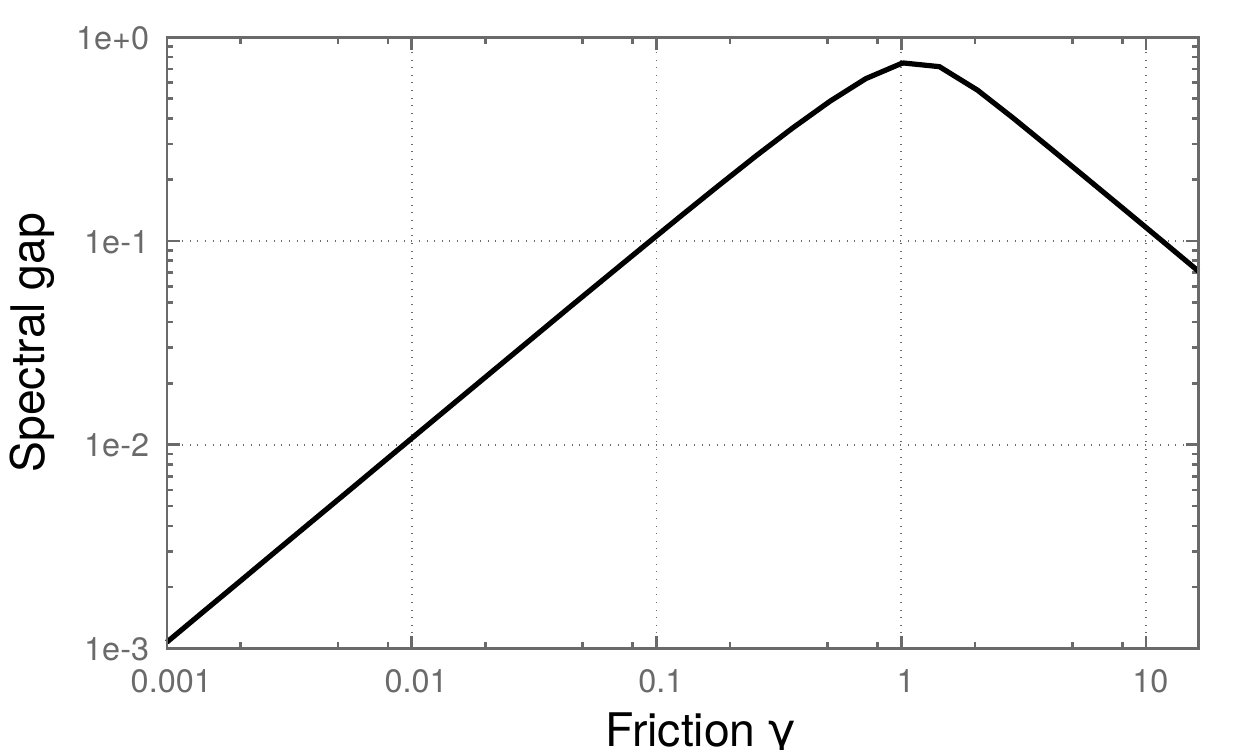}
\includegraphics[scale=.5]{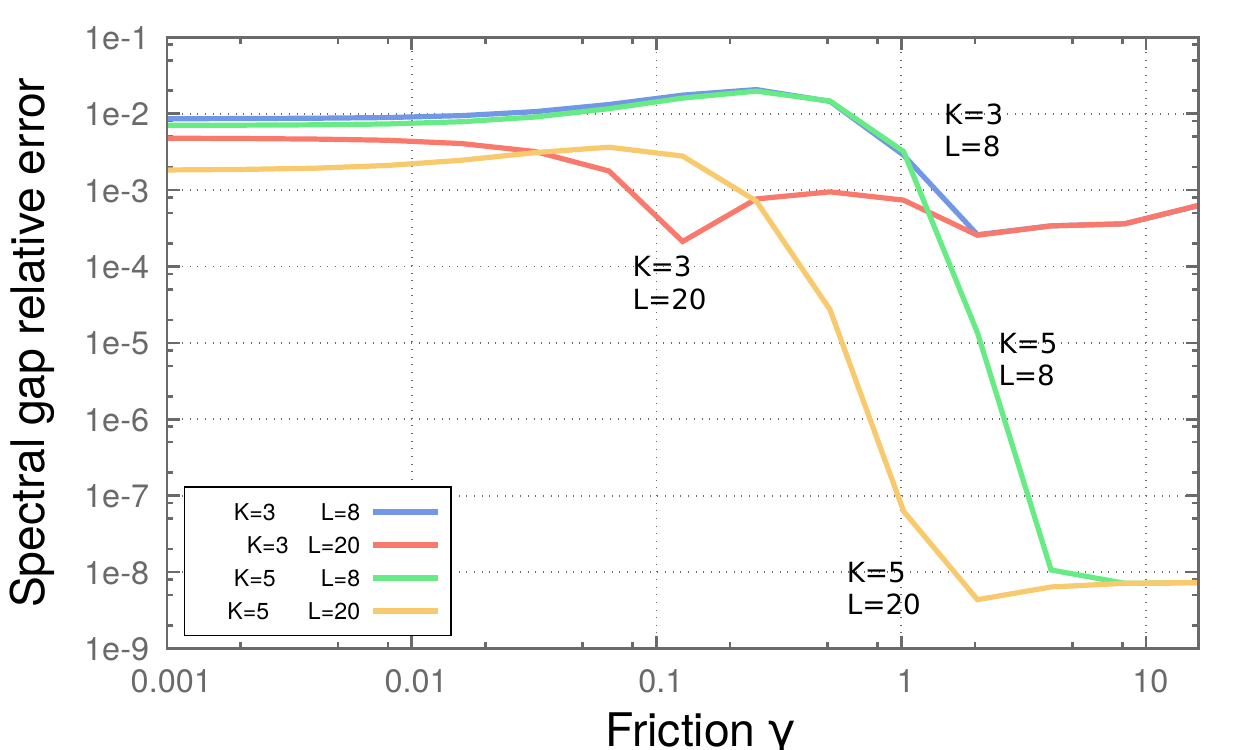}
\end{center}
\caption{Left: Spectral gap as a function of the friction~$\gamma$. Right: Relative error on the spectral gap as a function of~$\gamma$ for several couples $K,L$. Note that the curve corresponding to $K=3, L=8$ coincides with $K=5, L=8$ for $\gamma$ small and with $K=3, L=20$ for $\gamma$ large.}
\label{fig:spectral gap wrt gamma}
\end{figure}

\section*{Acknowledgements}

We thank Alexandre Ern, Tony Leli\`evre and François Madiot (CERMICS), as well as Greg Pavliotis and Urbain Vaes (Imperial College) for helpful discussions. This work is supported by the Agence Nationale de la Recherche under grant ANR-14-CE23-0012 (COSMOS); as well as the European Research Council under the European Union's Seventh Framework Programme (FP/2007-2013) -- ERC Grant Agreement number 614492. We also benefited from the scientific environment of the Laboratoire International Associ\'e between the Centre National de la Recherche Scientifique and the University of Illinois at Urbana-Champaign. Finally, we acknowledge the support from the International Centre for Theoretical Sciences (ICTS) for the program \emph{Non-equilibrium statistical physics} (ICTS/Prog-NESP/2015/10).

\appendix

\section{Proof of Theorem~\ref{th:hypocoercivity} ($\Lmu$ hypocoercivity)}
\label{app:proof hypocoercivity}

We recall in this section the proof of Theorem~\ref{th:hypocoercivity}, as presented in~\cite{Dolbeault09,Dolbeault15}. We start with the proofs of the technical results presented at the end of Section~\ref{s:equili}.

\begin{proof}[Proof of Lemma~\ref{lemma:equivalent norms}]
  Consider $\varphi \in \Pi_0 \core$. A simple computation shows that 
  \begin{equation}
    \label{eq:LhamPip}
  \Lham \Pip = \frac1\beta \naq \nap^* \Pip = \left(\frac{p}{m}\right)^\top \naq \Pip,
  \end{equation}
  which immediately implies that $\Lham \Pip \varphi$ has average~0 with respect to~$\kappa(\dd p)$ for any $q \in \calD$. Therefore, $\Pip \Lham \Pip = 0$, which implies $A = A(1-\Pip)$.
  
  By definition of the operator $A$, it also holds
  \[
  A \varphi + (\Lham \Pip)^* (\Lham \Pip) A \varphi = (\Lham \Pip)^* \varphi.
  \]
  This identity immediately implies that $\Pip A = A$. 
  Taking the scalar product with $A \varphi$, we obtain, using $\Lham A = \Lham \Pip A = (1-\Pip) \Lham A$:
  \begin{equation}
    \begin{aligned}
      \| A \varphi \|^2 + \| \Lham A \varphi \|^2 &= \lang \Lham A \varphi, \varphi \rang = \lang \Lham A \varphi, (1-\Pip) \varphi \rang \\
      & \leq \| (1-\Pip) \varphi \| \, \| \Lham A \varphi \| \\
      & \leq \frac 1 4 \| (1-\Pip) \varphi \|^2 + \| \Lham A \varphi \|^2.
    \end{aligned}
  \end{equation}
  The last inequality gives $\| A \varphi \| \leq \|(1-\Pi)\varphi \|/2$, while the second one implies that $\| \Lham A \varphi \| \leq \| (1-\Pip) \varphi \|$. The conclusion is finally obtained by density of~$\core$ in~$\Lmu$.
\end{proof}

The key element to prove Proposition~\ref{prop:coercivity_scrD} is the following coercivity estimates, respectively called ``microscopic'' and ``macroscopic'' coercivity in~\cite{Dolbeault09,Dolbeault15}. 

\begin{prop}[Coercivity properties]
  The operators $\LFD$ and $\Lham \Pip$ satisfy the following coercivity properties:
  \begin{equation}
    \label{eq:micro coercivity}
    \forall \varphi \in \core, \qquad - \lang \LFD \varphi, \varphi \rang \geq \frac 1 m \| (1 - \Pip) \varphi \|^2,
\end{equation}
  \begin{equation}
    \label{eq:macro coercivity}
    \forall \varphi \in \core \cap \tLmu, \qquad \| \Lham \Pip \varphi \|^2 \geq \frac {\Cnu}{\beta m} \| \Pip \varphi \|^2,
  \end{equation}
  where $\Cnu$ is defined in~\eqref{eq:poincare nu}. As a corollary, the following inequality holds in the sense of symmetric operators on $\tLmu$:
  \begin{equation}
    \label{eq:coercivity A Lham}
    A \Lham \Pip \geq \lambdaham \Pip, \qquad \lambdaham = 1 - \left(1 + \frac {\Cnu}{\beta m}\right)^{-1} > 0.
  \end{equation}
\end{prop}

\begin{proof}
The inequality~\eqref{eq:micro coercivity} directly results from a Poincaré inequality for the Gaussian measure~$\kappa$ (see~\cite{Beckner89}), the position $q$ being seen as a parameter. Indeed, for a given $\varphi \in \core$,
\begin{equation}
  \forall q \in \calD, \qquad \int_\bbRD \left| \nap \varphi(q,p) \right|^2 \, \kappa(\dd p) \geq \frac \beta m \int_\bbRD \left| (1 - \Pip) \varphi(q,p) \right|^2 \, \kappa(\dd p)
\end{equation}
Integrating against $\nu$ and noting that $- \lang \LFD \varphi, \varphi \rang = \betainv \| \nap \varphi \|^2$ leads to the desired inequality.

To prove~\eqref{eq:macro coercivity}, we use~\eqref{eq:LhamPip}, which leads to
\begin{equation}
  \| \Lham \Pip \varphi \|_\Lmu^2 = \frac 1 {\beta m} \| \nabla_q \Pip \varphi \|_\Lnu^2.
\end{equation}
The conclusion then follows from the Poincaré inequality~\eqref{eq:poincare nu}, since, for $\varphi \in \core \cap \tLmu$, the function $\Pip \varphi$ has average~0 with respect to~$\nu$ (namely, $\Esp_\nu[\Pip \varphi] = \Esp_\mu[\varphi] = 0$).

The macroscopic coercivity~\eqref{eq:macro coercivity} allows to write $(\Lham \Pip)^* (\Lham \Pip) \geq \frac {\Cnu}{\beta m} \Pip$ in the sense of symmetric operators on~$\tLmu$. Moreover,
\[
A \Lham \Pip = \left[ 1+(\Lham \Pip)^* (\Lham \Pip) \right]^{-1} (\Lham \Pip)^* (\Lham \Pip).
\]
Since $(\Lham \Pip)^* (\Lham \Pip)$ is self-adjoint and the function $x \mapsto x/(1+x) = 1 - 1/(1+x)$ is increasing, the inequality~\eqref{eq:coercivity A Lham} follows by spectral calculus.
\end{proof}

Another technical argument is the boundedness of certain operators, which appear in the proof of Proposition~\ref{prop:coercivity_scrD}.

\begin{lemma}
  \label{lemma:momentum operators}
  For any $\ell \in \bbN^*$, $i \in \{1,2,\dots,D\}$ and $\varphi \in \Lmu$,
  \[
  \| \Pip \papi^\ell \varphi \|_\Lnu \leq \sqrt{\left(\frac{\beta}{m}\right)^\ell \ell!} \, \| (1 - \Pip) \varphi \|.
  \]
  In particular, $\left\| \Pip \papi^\ell \right\| =  \left\| \left(\papi^*\right)^\ell \Pip \right\| \leq \sqrt{\beta^\ell \ell!}$. 
\end{lemma}

\begin{proof}
  Fix $\varphi \in \core$. For $q \in \calD$,
  \[
  \left(\Pip \papi^n \varphi\right) (q) = \int_{\bbR^D} \left(\papi^n (1-\Pip) \varphi\right)(q,p) \, \kappa(\dd p) = \int_{\bbR^D} (1-\Pip) \varphi(q,p) \, (\papi^*)^n \bfone \, \kappa(\dd p).
  \]
Denoting by $H_\ell(p_i) = (m/\beta)^{\ell/2} \ell!^{-1/2} (\papi^*)^\ell \bfone$ the Hermite polynomials in the variable~$p_i$ (which, we recall, are such that $\|H_\ell\|_{\Lkappa}=1$), 
a Cauchy--Schwarz inequality shows that
\[
\begin{aligned}
  \| \Pip \papi^\ell \varphi \|_\Lnu^2 &\leq \int_\calD \left( \int_{\bbR^D} | (1-\Pip) \varphi(q,p) | \, \left| \sqrt{\left(\frac{\beta}{m}\right)^\ell \ell!} H_\ell(p_i) \right| \kappa(\dd p) \right)^2 \nu(\dd q) \\
  &\leq \left(\frac{\beta}{m}\right)^\ell \ell! \int_\calD \| (1-\Pip) \varphi(q,\cdot) \|_\Lkappa^2 \| H_\ell \|_\Lkappa^2 \nu(\dd q) = \left(\frac{\beta}{m}\right)^\ell \ell! \| (1-\Pip) \varphi \|^2,
\end{aligned}
\]
which gives the claimed result.
\end{proof}

\begin{prop}[Boundedness of auxiliary operators]
  \label{prop:P4}
  There exist $R_{\rm ham}>0$ such that
  \begin{equation}
    \label{eq:P4}
    \forall \varphi \in \core, 
    \qquad
    \left\{ \begin{aligned}
    \|A \Lham (1-\Pip)\varphi \| & \leq R_{\rm ham} \|(1-\Pip) \varphi \|, \\ 
    \| A \LFD \varphi \| & \leq \frac{1}{2m}\|(1-\Pip) \varphi \|.
    \end{aligned} \right.
  \end{equation}
\end{prop}

\begin{proof}
The first task is to give a more explicit expression of the operator~$A$. In the following we use frequently the fact that operators acting only on the variables~$q$ (such as $\naq$ and $\naq^*$) commute with operators acting only on variables~$p$ (such as $\nap$, $\nap^*$ and $\Pip$). Moreover the relations $\papi \Pip = 0$, $\Pip \papi^* = 0$ and $\Pip \papi \papj^* = \papi \papj^* \Pip = \frac {\beta} m \Pip \delta_{ij}$ allow to simplify the action of $(\Lham \Pip)^* (\Lham \Pip)$ as follows:
\[
\begin{aligned}
  (\Lham \Pip)^* (\Lham \Pip) 
  & = -\frac{1}{\beta^2} \Pip (\nap^* \naq - \naq^* \nap ) ( \nap^* \naq - \naq^* \nap ) \Pip \\
  & = \frac{1}{\beta^2} \Pip (\naq^* \nap) (\nap^* \naq) \Pip = \frac 1 {\beta m} \naq^* \naq \Pip.
\end{aligned}
\]
The operator $A$ can therefore be reformulated as
\begin{equation}
  \label{eq:reformulate A}
  A = \frac1\beta \left(1 + \frac 1 {\beta m} \naq^* \naq \right)^{-1} \naq^* \Pip \nap.
\end{equation}
To obtain bounds on the operator $A \Lham (1-\Pip)$, we next consider its adjoint:
\[
\begin{aligned}
  -(1 - \Pip) \Lham A^*  
  & = - \frac{1}{\beta^2} (1-\Pip) \left( \nap^* \naq - \naq^* \nap \right) \nap^* \naq \Pip \left(1 + \frac 1 {\beta m} \naq^* \naq \right)^{-1} \\
  & = - \frac{1}{\beta^2} (1-\Pip) \left( \nap^* \naq \nap^* \naq - \frac{\beta}{m} \naq^* \naq\right) \Pip \left(1 + \frac 1 {\beta m} \naq^* \naq \right)^{-1} \\
  & = - \frac1{\beta^2} (1-\Pip) \nap^* \naq \nap^* \naq \Pip \left(1 + \frac 1 {\beta m} \naq^* \naq \right)^{-1},
\end{aligned}
\]
where we used $(1-\Pip) \naq^* \naq \Pip = 0$ in the last line. Moreover, the operator 
\[
\nap^* \naq \nap^* \naq \Pip = \sum_{i,j=1}^D \partial_{p_i}^* \partial_{p_j}^* \Pip \partial_{q_i} \partial_{q_j}
\]
is bounded from $\rmH^2(\nu)$ to $\Lmu$ according to~Lemma~\ref{lemma:momentum operators}. Moreover, as proved in~\cite{Dolbeault15}, 
Assumption~\ref{ass:potential} ensures that the operator $\Pip \left(1 + \frac 1 {\beta m} \naq^* \naq \right)^{-1}$ is bounded from $\Lmu$ to $\rmH^2(\nu)$. In conclusion, $-(1 - \Pip) \Lham A^*$ is bounded on $\Lmu$.

The boundedness of the operator $A \LFD$ comes from the fact that
\[
\begin{aligned}
  \Pip \Lham \LFD &= -\frac{1}{\beta^{2}} \Pip \left( \nap^* \naq - \naq^* \nap \right) \nap^* \nap = \frac{1}{\beta^{2}} \Pip \naq^* \nap \nap^* \nap \\
  & = \frac 1 {\beta m} \Pip \naq^* \nap = - \frac 1 m \Pip \Lham.
\end{aligned}
\]
In conclusion, $A \LFD = -A/m$, which gives the claimed result with Lemma~\ref{lemma:equivalent norms}. 
\end{proof}

We can now proceed with the proof of Proposition~\ref{prop:coercivity_scrD}. 

\begin{proof}[Proof of Proposition~\ref{prop:coercivity_scrD}]
Note first that, for a given $\varphi \in \core$, the entropy dissipation $\scrD[\varphi]$ can be explicitly written as
\begin{equation}
  \label{eq:expansion of D}
  \begin{aligned}
    \scrD[\varphi] &= \lang - \gamma \LFD \varphi, \varphi \rang + \varepsilon \lang A \Lham \Pip \varphi, \varphi \rang + \varepsilon \lang A \Lham (1-\Pip) \varphi, \varphi \rang \\
    &\ \ - \varepsilon \lang  \Lham A \varphi, \varphi \rang + \varepsilon \gamma \lang A \LFD \varphi, \varphi \rang,
  \end{aligned}
\end{equation}
since $\LFD A = \LFD \Pip A = 0$. Using respectively the properties~\eqref{eq:micro coercivity}, \eqref{eq:coercivity A Lham}, \eqref{eq:P4} and Lemma~\ref{lemma:equivalent norms}, it follows
\begin{equation} 
\label{eq:bound on D}
\begin{aligned}
\scrD[\varphi] & \geq \frac \gamma {m} \| (1 - \Pip) \varphi \|^2 + \varepsilon \lambdaham \| \Pip \varphi \|^2 - \varepsilon \left(R_{\rm ham} + \frac{\gamma}{2m}\right)\|(1 - \Pip) \varphi \| \, \| \Pip \varphi \| \\
& \ \ - \varepsilon \lang  \Lham A \varphi, \varphi \rang.
\end{aligned} \end{equation}
Since, by Lemma~\ref{lemma:equivalent norms},
\[
\lang \Lham A \varphi, \varphi \rang = \lang (1-\Pip) \Lham \Pip A (1-\Pip) \varphi, \varphi \rang \leq \|(1 - \Pip) \varphi \|^2,
\]
it holds $\scrD[\varphi] \geq X^\top \bfS X$, where
\[
X = \begin{pmatrix} \| \Pip \varphi\| \\ \| (1-\Pip) \varphi\| \end{pmatrix}, 
\qquad 
\bfS = \begin{pmatrix} S_{--} & S_{-+}/2 \\ S_{-+}/2 & S_{++} \end{pmatrix},
\]
with
\[
S_{--} = \varepsilon \lambdaham,
\qquad
S_{-+} = - \varepsilon \left(R_{\rm ham} + \frac{\gamma}{2m}\right),
\qquad 
S_{++} = \frac \gamma {m} - \varepsilon. 
\]
The smallest eigenvalue of $\bfS$ is 
\[
\Lambda(\gamma,\varepsilon) = \frac {S_{--}+S_{++}} 2 - \frac 1 2 \sqrt{(S_{--}-S_{++})^2+(S_{-+})^2}.
\]
In the limit $\gamma \to 0$, the parameter $\varepsilon$ should be chosen of order~$\gamma$ in order for $\Lambda(\gamma,\varepsilon)$ to be positive (in particular for $S_{++}$ to remain positive). When $\gamma \to +\infty$, the parameter $\varepsilon$ should be chosen of order~$1/\gamma$ in order for the determinant of $\bfS$ to remain positive. We therefore consider the choice
\begin{equation}
  \label{eq:choice_varepsilon}
  \varepsilon = \overline{\varepsilon} \min(\gamma,\gamma^{-1}).
\end{equation}
It is then easy to check that there exists $\overline{\varepsilon} > 0$ sufficiently small such that $\Lambda(\gamma,\overline{\varepsilon} \min(\gamma,\gamma^{-1})) > 0$ for all $\gamma > 0$. Moreover, it can be proved that $\Lambda(\gamma,\overline{\varepsilon} \min(\gamma,\gamma^{-1}))/\gamma$ converges to a positive value as $\gamma \to 0$, while $\gamma \Lambda(\gamma,\overline{\varepsilon} \min(\gamma,\gamma^{-1}))$ converges to a positive value as $\gamma \to +\infty$. This gives the claimed result with $\widetilde{\lambda}_\gamma = \Lambda(\gamma,\overline{\varepsilon} \min(\gamma,\gamma^{-1}))$.
\end{proof}

The proof of Theorem~\ref{th:hypocoercivity} is now easy to obtain. Consider $\varphi_0 \in \rmDom(\calL) \cap \tLmu$ (which contains $\rmH^2(\mu) \cap \tLmu$) and introduce $\scrH(t) = \calH[\varphi(t)]$, where $\varphi(t) = \rme^{t \calL} \varphi_0 \in \rmDom(\calL)$ for any $t \geq 0$. Then, 
\[
\scrH'(t) = -\scrD[\varphi(t)] \leq -\widetilde{\lambda}_\gamma \| \varphi(t) \|^2.
\]
Using the norm equivalence~\eqref{eq:equivalent norms} and the choice~\eqref{eq:choice_varepsilon} for $\overline{\varepsilon}<1$, it follows that 
\[
\scrH'(t) \leq - \frac{2\widetilde{\lambda}_\gamma}{1+\overline{\varepsilon}\min(\gamma,\gamma^{-1})} \scrH(t),
\]
so that, by a Gronwall estimate,
\[
\scrH(t) \leq \scrH(0) \exp\left(- \frac{2\widetilde{\lambda}_\gamma}{1+\overline{\varepsilon}\min(\gamma,\gamma^{-1})} t\right). 
\]
Using again the norm equivalence~\eqref{eq:equivalent norms}, it follows that 
\[
\|\varphi(t)\|^2 \leq \frac{1+\overline{\varepsilon}}{1-\overline{\varepsilon}} \, \rme^{-2\lambda_\gamma t} \|\varphi(0)\|^2, 
\]
with the decay rate
\[
\lambdagamma = \frac{\widetilde{\lambda}_\gamma}{1+\overline{\varepsilon}\min(\gamma,\gamma^{-1})}. 
\]
The desired estimate finally follows by density of $\rmDom(\calL)$ in $\Lmu$.

\section{Proof of technical estimates for the system considered in Section~\ref{s:eq appli}}
\label{app:proof discrete hypocoercivity}

We prove in this section that the conditions~\eqref{eq:additional conditions non-conformal} and~\eqref{eq:conditions hypocoercivity} allowing to apply the results of Section~\ref{s:discrete convergence} hold for the system considered in Section~\ref{s:eq appli}. Recall that the condition $M \to +\infty$ should be understood as $K,L \to +\infty$. Let us also emphasize that, although we perform the computations for the simple potential $V(q) = 1-\cos(q)$, the extension to a general trigonometric polynomial~$V$ is straightforward.

\paragraph{Condition~\eqref{eq:additional conditions non-conformal} and bound on $\|\calL u_K\|$.}
Since $u_M$ depends only on the positions, it is denoted $u_K$ and
\[
\| \calL u_K \|^2 = \| \calL^* u_K \|^2 = \frac{1}{\| \PiK \bfone \|^2} \left\| p \paq \PiK \bfone \right\|^2= \beta \frac {\| \paq \PiK \bfone \|^2}{\| \PiK \bfone \|^2}. 
\]
In order to estimate $\| \paq \PiK \bfone \|$, we decompose $\PiK \bfone$ in the basis under consideration as follows: 
\[
\PiK \bfone = \sum_{j=0}^{2K-2} g_j G_j, \qquad g_j = \lang \PiK \bfone, G_j \rang = \int_\calD G_j \, d\nu. 
\]
Then, using $\paq \PiK \bfone = -\paq (1-\PiK) \bfone$ and (with~\eqref{eq:derivatives G})
\begin{equation}
  \label{eq:derivative_*_q_Gk}
  \forall k \geq 1, 
  \qquad
  \partial_q^* G_{2k} = -\frac{\beta}{4} G_{2k-3} + k G_{2k-1} + \frac{\beta}{4} G_{2k+1},
  \qquad
  \partial_q^* G_{2k-1} = \frac{\beta}{4} G_{2k-2} - k G_{2k} - \frac{\beta}{4} G_{2k+2},
\end{equation}
it follows that, for $K \geq 1$,
\[
\begin{aligned}
  \| \paq \PiK \bfone \|^2 &= \sum_{j \in \bbN} \lang \paq \PiK \bfone, G_{j} \rang^2\\
  &= \sum_{j=0}^{2K-2} \lang -\paq (1-\PiK) \bfone, G_{j} \rang^2 + \sum_{j=2K-1}^{+\infty} \lang \paq \PiK \bfone, G_{j} \rang^2 \\
  &= \sum_{j=0}^{2K-2} \Esp_\nu\left[(1-\PiK) \paq^* G_{j} \right]^2 + \sum_{j=2K-1}^{+\infty} \Esp_\nu\left[\PiK \paq^* G_{j} \right]^2 \\
  & = \Esp_\nu[(1-\PiK) \paq^* G_{2K-3}]^2 + \Esp_\nu[(1-\PiK) \paq^* G_{2K-2}]^2 + \Esp_\nu[\PiK \paq^* G_{2K-1}]^2 + \Esp_\nu[\PiK \paq^* G_{2K}]^2\\
  &= \frac {\beta^2} {16} \left( g_{2K}^2+g_{2K-1}^2+g_{2K-2}^2+g_{2K-3}^2 \right) \leq \frac {\beta^2} {16} \left\| (1-\Pi_{K-1}^q) \bfone \right\|^2.
\end{aligned}
\]
Since $\bfone \in \Hsnu$ for any $s \in \bbN$, it follows that $\| (1-\Pi_{K-1}^q) \bfone \|$ vanishes faster than any polynomial in~$K$ in view of Lemma~\ref{lem:approx_PiK}. This implies that $\| \paq \PiK \bfone \|$, and hence $\| \calL u_K \|$ and $\| \calL^* u_K \|$, vanish faster than any polynomial in~$K$. More precisely, 
\begin{equation}
  \label{eq:L u_M}
  \| \calL^* u_K \|^2 = \| \calL u_K \|^2 \leq \frac{\beta^{3}} {16} \frac{\left\| (1-\Pi_{K-1}^q) \bfone \right\|^2}{\|\Pi_K^q \bfone\|^2} \leq \frac{\beta^{3}} {16} \frac{\left\| (1-\Pi_{K-1}^q) \bfone \right\|^2}{1 - \|(1-\Pi_K^q) \bfone\|^2}.
\end{equation}

\paragraph{Condition~\eqref{eq:conditions hypocoercivity}.}
\newcommand{\Lpm}{\calL_{KL}^{+-}}

Let us now prove that $\| (A+A^*) (1-\PiKL) \calL \PiKL \| \xrightarrow[K,L \to \infty]{} 0$ for the model under consideration. Introducing $\Lpm = (1-\PiKL) \calL \PiKL$, we prove in fact that $A\Lpm$ and $A^*\Lpm$ are bounded operators whose norms converge to~0 as $K,L \to +\infty$. In all this proof, we consider $K \geq 1$ and $L\geq 2$.

The first task is to provide a more explicit expression of $\Lpm$. We introduce to this end the operator $\Pif = \PiKo \paq \PiK$. In view of~\eqref{eq:derivatives G},
\[
\begin{aligned}
  \Pif \varphi &= \sum_{j' = 2K-1}^{+\infty} \sum_{j = 0}^{2K-2} \lang \varphi, G_j \rang \lang \paq G_j, G_{j'} \rang G_{j'} 
  = \frac \beta 4 \Big( \lang \varphi, G_{2K-2} \rang G_{2K-1} - \lang \varphi, G_{2K-3} \rang G_{2K} \Big).
\end{aligned}
\]
This shows that the operator $\Pif$ is bounded on $\Lmu$, and in fact 
\begin{equation}
  \label{eq:Pif}
  \| \Pif \varphi \| \leq \frac \beta 4 \| \PiKpo \PiK \varphi \|.
\end{equation}
Comparing~\eqref{eq:derivative_*_q_Gk} and~\eqref{eq:derivatives G}, we also see that $\Pif = \PiKo \paq \PiK = \PiKo \paq^* \PiK$. We can now compute more explicitly the action of $\Lpm$ by noting that
\[
\beta \Lpm = (1-\PiKL) \paq \pap^* \PiKL - (1-\PiKL) \paq^* \pap \PiKL - \gamma (1-\PiKL) \pap^* \pap \PiKL,
\]
where $(1-\PiKL) \pap^* \pap \PiKL = 0$ by~\eqref{eq:LFD hermite}, while (using~\eqref{eq:derivatives H} to write $\Pi_{L-1}^p \partial_p = \partial_p \Pi_L^p$ and $\Pi_{L+1}^p \partial_p^* = \partial_p^* \Pi_L^p$)
\[
\begin{aligned}	
  (1-\PiKL) \paq \pap^* \PiKL &= (1-\PiK \PiL) \paq \PiLp \pap^* \PiK \\
  & = (1-\PiK \PiL) \paq (\PiL + \PiLo) \PiLp \pap^* \PiK \\
  & = (\PiL + \PiLo \PiLp -\PiK \PiL) \paq \pap^* \PiK \\
  & = \PiL (1-\PiK) \paq \pap^* \PiK + \paq \PiLp \PiLo \pap^* \PiK \\
  & = \PiL \pap^* \Pif + \paq \pap^* \PiL \PiLmo \PiK \\
  &= \pap^* \PiLm \Pif + \paq \pap^* \PiLmo \PiKL, 
\end{aligned}
\]
and 
\[
\begin{aligned}
  (1-\PiKL) \paq^* \pap \PiKL  & = \pap (1-\PiK \PiLp) \paq^* \PiL \PiK = \pap \PiL (1-\PiK) \paq^* \PiK \\
  & = \Pif \pap \PiL.
\end{aligned}
\]
Therefore, 
\begin{equation}
  \label{eq:simpler_expression_Lpm}
  \beta \Lpm = \pap^* \PiLm \Pif + \paq \pap^* \PiLmo \PiKL - \Pif \pap \PiL.
\end{equation}
Moreover $\| \partial_p^* \Pi_{L-1}^p\| \leq \sqrt{\beta (L-1)}$, $\| \partial_p \Pi_L^p\| \leq \sqrt{\beta (L-1)}$ and using the Gerschgorin theorem (see~\cite{Qi84} for example) $\|\paq \Pi_K^q\| \leq K-1+\beta/2$, so the operator $\Lpm$ is bounded, with 
\begin{equation}
\begin{aligned}
  \label{eq:norm Lpm}
  \left\| \Lpm \right\| &\leq \betainv \sqrt{\beta (L-1)} \frac \beta 4 + \betainv \sqrt{\beta L} \left( K-1+\frac \beta 2 \right) + \betainv \sqrt{\beta (L-1)} \frac \beta 4 \\
  &\leq \sqrt{\frac{L}{\beta}} \left( K-1+\beta \right).
\end{aligned}
\end{equation}
We are now in position to provide a more explicit expression of $A \Lpm$ and $A^* \Lpm$ based on~\eqref{eq:simpler_expression_Lpm}. Recalling the definition~\eqref{eq:def_Pip} of $\Pip = \Pi_1^p$, it holds $\Pip \, \PiLmo = 0$ and $\Pip \, \PiLm = \Pip$ for $L \geq 2$. Using also the relation $\Pip \pap \pap^* = \beta$, we obtain
\[
\begin{aligned}
  (\Lham \Pip)^* \Lpm &= \betainv \Pip \paq^* \pap \Lpm \\
  &= \betainv \Pip \paq^* \PiLm \Pif + \betainv \Pip \paq^* \paq \PiLmo \PiKL  - \betainvinv \Pip \paq^* \pap^2 \Pif \PiL \\
  &= \betainv \Pip \paq^* \Pif - \betainvinv \Pip \paq^* \pap^2 \Pif \\
  &= \betainv \Pip\left(1 - \betainvinv \pap^2 \right) \paq^* \Pif
\end{aligned}
\]
since $L \geq 2$. Introducing the generator of the overdamped Langevin dynamics (for $m=1$ here)
\[
\Lovd = -\betainv \paq^* \paq,
\]
it is possible to rewrite~\eqref{eq:reformulate A} as $A = \left( 1 - \Lovd \right)^{-1} \Pip \pap \paq^*$, so that
\begin{equation}
  \label{eq:first rest term}
  A \Lpm = \left( \betainv \Pip - \betainvinv \Pip \pap^2 \right)  (1-\Lovd)^{-1} \paq^* \Pif.
\end{equation}
Similar computations show that (using $ \Pip \pap^* = 0$)
\begin{equation}
  \label{eq:second rest term}
  \begin{aligned}
    A^* \Lpm 
    &= -\beta^{-2} \pap^* \paq (1-\Lovd)^{-1} \Pip \pap \Pif \PiL \\
    &= -\betainvinv \pap^* \Pip \pap \paq (1-\Lovd)^{-1} \Pif.
\end{aligned}
\end{equation}
The momentum operators $\Pip$, $\Pip \pap^2$ and $\pap^* \Pip \pap$ are bounded according to Lemma~\ref{lemma:momentum operators}:
\[
\left\| \betainvinv \Pip \pap^2 - \betainv \Pip \right\|_{\calB(\Lkappa)} \leq \frac{\sqrt 2 + 1}{\beta}, 
\qquad
\left\| \pap^* \Pip \pap \right\|_{\calB(\Lkappa)} \leq \beta,
\]
so that
\begin{equation}
\label{eq:sp gap error reformulation}
\begin{aligned}
  \left\| A \Lpm \right\|_{\calB(\Lmu)} &\leq \frac{\sqrt 2 + 1}{\beta} \left\| (1-\Lovd)^{-1} \paq^* \Pif \right\|_{\calB(\Lnu)}, \\
  \left\| A^* \Lpm \right\|_{\calB(\Lmu)} &\leq \frac1\beta \left\| \paq (1-\Lovd)^{-1} \Pif \right\|_{\calB(\Lnu)}.
\end{aligned}
\end{equation}

At this stage, it remains to prove that the operators on $\Lnu$ in the right-hand sides of the previous inequalities are bounded, with vanishing norms as $K \to +\infty$. We use to this end the following decompositions: 
\[
(1-\Lovd)^{-1} \paq^* \Pif = T_1 S_{1,K} \Pif, \qquad \paq (1-\Lovd)^{-1} \Pif = T_2 S_{2,K} \Pif,
\]
with (using $\Pif = \PiKmo \, \Pif$)
\begin{equation}
  \begin{aligned}
    T_1 & = (1-\Lovd)^{-1} \paq^* (1 - \tLovd)^{\half}, & \qquad S_{1,K} = (1 - \tLovd)^{-\half} \PiKmo, \\
    T_2 & = \paq  (1-\Lovd)^{-\half}, & \qquad S_{2,K} = (1 - \Lovd)^{-\half} \PiKmo,
  \end{aligned} 
\end{equation}
where we introduced the symmetric negative operator $\tLovd = -\betainv \paq \paq^*$. Let us show that $T_1$ and $T_2$ are bounded and $S_{1,K}$ and $S_{2,K}$ can be made small for $K$ sufficiently large. Note first that 
\[
\begin{aligned}
T_1 T_1^* & = (1-\Lovd)^{-1} \paq^* \left(1 - \tLovd\right) \paq (1-\Lovd)^{-1} \\
& = (1-\Lovd)^{-1} \left( \paq^* \paq + \betainv \paq^* \paq \paq^* \paq \right) (1-\Lovd)^{-1} = - \beta  (1-\Lovd)^{-1} \Lovd,
\end{aligned}
\]
so that, by spectral calculus, $0 \leq T_1 T_1^* \leq \beta$. This shows that $T_1^*$ and $T_1$ are bounded operators on $\Lnu$, with $\|T_1^*\|=\|T_1\| \leq \sqrt{\beta}$. Similarly, 
\[
T_2^* T_2 = -\beta (1-\Lovd)^{-\half} \Lovd (1-\Lovd)^{-\half},
\]
from which we deduce $\|T_2^*\|=\|T_2\| \leq \sqrt{\beta}$. We next prove that the operators $S_{1,K}$ and $S_{2,K}$ can be made as small as wanted by increasing $K$. We start by proving the following lemma.

\begin{lemma}
  \label{lem:ineq_1_Lovd}
  For $K \geq 2$, the following inequalities hold in the sense of symmetric operators:
  \[
  1 - \Lovd \geq \betainv (K-1)^2 \PiKmo, 
  \qquad 
  1 - \tLovd \geq \betainv (K-1)^2 \PiKmo.
  \]
\end{lemma}

\begin{proof}
The operator $1-\Lovd$ can be expressed in the $\Lmu$-orthonormal basis $G_k$ as
\begin{equation}
  \label{eq:expression_Lovd_Gk_basis}
  \left\{ \begin{aligned}
    (1-\Lovd) G_{2k-1} &= -\frac \beta {16} (G_{2k-5} + G_{2k+3}) - \frac 1 4 (G_{2k-3} + G_{2k+1}) + \left( 1 + \frac \beta 8 + \frac{k^2}{\beta} \right) G_{2k-1}, \\
    (1-\Lovd) G_{2k} &= -\frac \beta {16} (G_{2k-4} + G_{2k+4}) - \frac 1 4 (G_{2k-2} + G_{2k+2}) + \left( 1 + \frac \beta 8 + \frac{k^2}{\beta} \right) G_{2k}.
  \end{aligned} \right.
\end{equation}
Similar formulas hold for $1-\tLovd$, upon changing the factors $-1/4$ into $1/4$ in the above expressions. Therefore, the symmetric operators $1-\Lovd - \left( \betainv (K-1)^2 + \frac 1 2 \right) \PiKmo$ and $1-\tLovd - \left( \betainv (K-1)^2 + \frac 3 2 \right) \PiKmo$ can be represented by diagonally dominant matrices in the basis $(G_k)$, which shows that these operators are positive.
\end{proof}

\newcommand{\calAmm}{\calA^{--}}
\newcommand{\calAmp}{\calA^{-+}}
\newcommand{\calApm}{\calA^{+-}}
\newcommand{\calApp}{\calA^{++}}

\begin{lemma}
  \label{lem:bounds_A++^-1}
There exists $K_0 \in \bbN$ such that, for any $K \geq K_0$, the following inequalities hold in the sense of symmetric operators: 
\[
0 \leq \PiKmo (1 - \Lovd)^{-1} \PiKmo \leq \frac{2 \beta}{K^{2}}, 
\qquad 
0 \leq \PiKmo \left(1 - \tLovd\right)^{-1} \PiKmo \leq \frac{2 \beta}{K^{2}}.
\]
\end{lemma}

\begin{proof}
  We write the proof for the operator $\calA = 1 - \Lovd$, the result for $1-\tLovd$ being obtained by similar manipulations. Consider the following block decomposition with respect to $\PiKmo$ for $K$ fixed: 
  \[
  \calA = \begin{pmatrix}
    \calAmm & \calAmp \\ \calApm & \calApp
  \end{pmatrix}.
  \]
  More precisely, $\calAmm = \PiKm \calA \PiKm$, $\calAmp = \PiKm \calA \PiKmo$, $\calApm = \PiKmo \calA \PiKm$ and $\calApp = \PiKmo \calA \PiKmo$. A similar decomposition holds for $\calA^{-1}$. With this notation, the goal is to estimate $\left( \calA^{-1} \right)^{++} = \PiKmo (1 - \Lovd)^{-1} \PiKmo$. By the Schur complement formula, 
  \[
  \left( \calA^{-1} \right)^{++} = \left[ \calApp - \calApm \left( \calAmm \right)^{-1} \calAmp \right]^{-1},
  \]
  provided the operators under consideration are all invertible. By Lemma~\ref{lem:ineq_1_Lovd},
  \[
  \calApp - \calApm \left(\calAmm\right)^{-1} \calAmp \geq \left( \frac{(K-1)^2}{\beta}  - \| \calApm \|^2 \left\| \left( \calAmm \right)^{-1} \right\| \right) \PiKmo.
  \]
  Since $\left({\calAmm}\right)^{-1} \leq 1$ (because $\calAmm \geq 1$) and, in view of~\eqref{eq:expression_Lovd_Gk_basis}, 
  \[
  \| \calApm \|^2 \leq \frac 1 8 + \frac {\beta^2} {64},
  \]
  the Schur complement is invertible for $K$ sufficiently large, and its inverse is a symmetric operator satisfying
  \[
  0 \leq \left( \calA^{-1} \right)^{++} \leq \left[ \frac{(K-1)^2}{\beta} - \left( \frac 1 8 + \frac {\beta^2} {64} \right) \right]^{-1} \PiKmo. 
  \]
  The right-hand side is, in turn, smaller than $2 \beta/K^{2}$ for $K \geq K_0$ with $K_0$ sufficiently large. 
\end{proof}

Since $S_{2,K}^* S_{2,K} = \PiKmo (1 - \Lovd)^{-1} \PiKmo$ and $S_{1,K}^* S_{1,K} = \PiKmo (1 - \tLovd)^{-1} \PiKmo$, Lemma~\ref{lem:bounds_A++^-1} immediately implies that
\begin{equation}
  \forall K \geq K_0, \qquad 
  \| S_{1,K} \|_\Lnu \leq \frac{\sqrt{2\beta}}{K},
  \qquad
  \| S_{2,K} \|_\Lnu \leq \frac{\sqrt{2\beta}}{K}.
\end{equation}
The conclusion now follows from~\eqref{eq:Pif} (which implies that $\left\| \Pif \right\|_\Lnu \leq \beta/4$) and~\eqref{eq:first rest term}-\eqref{eq:second rest term}, which lead to 
\[
\left\| T_1 S_{1,K} \Pif \right\|_\Lnu \leq \frac {\sqrt 2 \beta^2} {4 K},
\qquad
\left\| T_2 S_{2,K} \Pif \right\|_\Lnu \leq \frac {\sqrt 2 \beta^2} {4 K}.
\]
Using~\eqref{eq:sp gap error reformulation}, we finally obtain
\[
\left\| (A+A^*) \Lpm \right\|_{\calB(\Lmu)} \leq \frac {(1+ \sqrt 2) \beta} {2 K}.
\]

\paragraph{Final explicit estimates.}
Using the bounds provided in this appendix, it is easily seen that the constant $\tlambdagammaKL$ introduced in Corollary~\ref{coro:non-conformal hypocoercivity} satisfies~\eqref{eq:estimate_tlambdagammaM_explicit}. It is then possible to make explicit the resolvent bound~\eqref{eq:bound inverse tcalL}.

\bibliographystyle{abbrv}
\bibliography{bibliography}

\end{document}